\documentclass[12pt]{amsart}

\usepackage[dvipsnames]{xcolor}
\usepackage{mathtools}
\usepackage{graphicx}
\usepackage{amsmath}
\usepackage{color}
\usepackage{caption}
\usepackage{subcaption}
\usepackage{amsfonts,amssymb,amscd}
\usepackage{epstopdf}
\usepackage{mathrsfs}

\calclayout

\newtheoremstyle{remboldstyle}
  {}{}{\itshape}{}{\bfseries}{.}{.5em}{{\thmname{#1 }}{\thmnumber{#2}}{\thmnote{ (#3)}}}
\theoremstyle{remboldstyle}

\newtheorem{thm}{Theorem}[section]
\newtheorem{prop}[thm]{Proposition}
\newtheorem{lem}[thm]{Lemma}
\newtheorem{cor}[thm]{Corollary}

\theoremstyle{definition}
\newtheorem{definition}[thm]{Definition}

\newtheorem{rem}[thm]{Remark}
\newtheorem{notation}[thm]{Notation}

\newtheorem{thmx}{Theorem}

\numberwithin{equation}{section}

\def\Bbb{\mathbb}
\def\reals{\Bbb R}

\def\disk{\Bbb D}

\newcommand{\Chat}{\widehat{\mathbb{C}}}

\DeclareMathOperator{\dist}{dist} 

\begin{document}
\epstopdfsetup{outdir=./}


\title[A Geometric Approach to Polynomial and Rational Approximation]{A Geometric Approach to Polynomial and Rational Approximation}



\author {Christopher J. Bishop}
\address{C.J. Bishop\\
         Mathematics Department\\
         Stony Brook University \\
         Stony Brook, NY 11794-3651}
\email {bishop@math.stonybrook.edu}
\author{Kirill Lazebnik}
\address{Kirill Lazebnik\\
Mathematics Department \\
University of North Texas \\
Denton, TX, 76205}
\email{Kirill.Lazebnik@unt.edu}
\thanks{\noindent       The  first author is partially supported       by NSF Grant DMS 1906259.       }





\begin{abstract} We strengthen the classical approximation 
theorems of Weierstrass, Runge and Mergelyan
by showing the polynomial 
and rational approximants can be taken to have a simple geometric structure. In particular, when approximating a function $f$ on a compact set $K$,  the critical points of our approximants may be taken to lie in any given domain containing $K$, and all the critical values in any given neighborhood of the polynomially convex hull of $f(K)$.\end{abstract}


\maketitle

\vspace{-2mm}


\vspace{-10mm}

\section{Introduction}

The following is Runge's classical theorem on polynomial approximation.

\begin{thm}\label{classical_runge} \cite{MR1554664} Let $f$ be a function analytic on a neighborhood of a compact set $K\subset\mathbb{C}$, and suppose $\mathbb{C}\setminus K$ is connected. For all $\varepsilon>0$, there exists a polynomial $p$ so that 
\begin{equation}\nonumber ||f-p||_K := \sup_{z\in K}|f(z)-p(z)| <\varepsilon. \end{equation}
\end{thm}

This famous result does not say much about what the polynomial approximant $p$ looks like off the compact set.  For various applications, it would be useful to understand the global behavior of $p$ and, in particular, the location of the critical points and values of $p$. To this end, we state our first result (Theorem \ref{main_theorem} below) after introducing the following notation.

\begin{notation} For any compact set $K\subset\mathbb{C}$ we
denote the $\varepsilon$-neighborhood of $K$ by $N_\varepsilon K:=\{z : \inf_{w\in K}|z-w|<\varepsilon\}$, and we denote by $\widehat{K}$ the union of $K$ with
all bounded components of $\mathbb{C}\setminus K$ (this is usually called the \emph{polynomially convex hull} of $K$: see \cite{MR1482798}). 
We say $K$ is \emph{full} if $\mathbb{C}\setminus K$ is connected.
We let  $\textrm{CP}(f)$ denote the set of critical points of
an analytic function $f$, and  let $\textrm{CV}(f):=f(\textrm{CP}(f))$ 
denote its critical values. 
A \emph{domain} in $\mathbb{C}$ is an open, connected subset of $\mathbb{C}$. 
\end{notation}


\begin{thmx}\label{main_theorem} {\bf (Polynomial Runge$+$)} \emph{ Let $K\subset\mathbb{C}$ be compact and full, $\mathcal{D}$ a domain containing $K$, and suppose $f$ is a function analytic in a neighborhood of $K$. Then for all $\varepsilon>0$, there exists a polynomial $p$ so that $||p-f||_{K}<\varepsilon$  and: }
\begin{enumerate}
 \item $\textrm{CP}(p)\subset \mathcal{D}$, 
 \item $\textrm{CV}(p) \subset N_\varepsilon\widehat{f(K)}$.
\end{enumerate} 
\end{thmx}

We remark that no relation is assumed between the domain $\mathcal{D}\supset K$ and the neighborhood of $K$ in which $f$ is analytic. Analogous improvements of the polynomial approximation
theorems of Mergelyan and Weierstrass will be stated and 
proved in Section \ref{main_proofs_section} (see Theorem \ref{Mergelyan} and Corollary \ref{Weierstrass}). 
When $K$ is not full, uniform approximation by polynomials is
not always possible, and so we turn to rational approximation. 
We denote the  Hausdorff distance between two sets $X$, $Y$, 
by $\textrm{d}_H(X,Y)$.

\begin{thmx}\label{rational_runge} {\bf (Rational Runge$+$)} 
\emph{ Let $K\subset\mathbb{C}$ be compact, $\mathcal{D}$ a domain containing $K$, $f$ a function analytic in a neighborhood of $K$, and suppose $P\subset\Chat\setminus K$ contains exactly one point from each component of $\Chat\setminus K$. Then there exists $P'\subset P$ so that for all $\varepsilon>0$, there is a rational function $r$ so that $||r-f||_{K}<\varepsilon$ and: }
\begin{enumerate}
\item $\textrm{d}_H(r^{-1}(\infty),P')<\varepsilon$ and $|r^{-1}(\infty)|=|P'|$,
 \item $\textrm{CP}(r)\subset \mathcal{D}$, 
 \item $\textrm{CV}(r) \subset N_\varepsilon\widehat{f(K)}$.
\end{enumerate} 
\end{thmx}

The behavior of $p$ off $K$ is of particular interest
in applications, such as in complex dynamics where
approximation results have been used to prove 
the existence of various dynamical behaviors for
entire functions (see, for example, 
\cite{EL87}, \cite{2020arXiv201114736E}, 
\cite{MR4375923},  
\cite{2021arXiv210810256M}, \cite{2022arXiv220411781M}, \cite{MR4458398}).
However, not understanding the critical points and values of $p$ means 
it has not been known whether these behaviors can occur
within restricted classes of entire functions, such as
the well studied Speiser or Eremenko-Lyubich classes 
(see the survey \cite{MR3852466}). 

We now briefly describe our approach. Let $\Omega\subset\mathbb{C}$ be a finitely connected domain with analytic boundary, and $f: \Omega \rightarrow \mathbb{D}$ analytic. By a theorem of Grunsky, $f$ can be approximated on any compact subset of $\Omega$ by a \emph{proper} holomorphic map $B: \Omega\rightarrow\mathbb{D}$ (recall proper means that the continuous extension of $B$ to $\partial\Omega$ satisfies $B(\partial\Omega)=\mathbb{T}:=\{z: |z|=1\}$). Grunsky's proof (see Lemma 4.5.4 of \cite{MR0463413}) uses a Riemann sum to approximate an integral representation of $f$ involving the Green's function on $\Omega$. Another approach can be found in \cite{Khavinsonref}. When $\Omega$ is simply connected, $B$ is a Blaschke product (up to a change of coordinates), and in this case the result is due to Carath\'eodory \cite{MR0064861}, with a much simpler proof based on power series. We provide several refinements of these results in \cite{BL_in_prep}, although the theorems of Grunsky and Carath\'eodory will suffice for the purposes of this manuscript.




Thus the function $f$ in Theorem \ref{main_theorem} or \ref{rational_runge} can be approximated on $K$ by a holomorphic map $B$ defined in a union $D:=\cup_iD_i\supset K$ of pairwise disjoint domains $(D_i)_{i=1}^k$, so that $B$ is proper in each $D_i$. Our approach in this manuscript is to extend $B$ from $D$ to a quasiregular mapping $g:\Chat\rightarrow\Chat$ with specified poles. The Measurable Riemann Mapping Theorem (MRMT for brevity) will then imply that there is a quasiconformal mapping $\phi$ so that $g\circ\phi^{-1}$ is rational. The bulk of the work in this paper will be to show that there exists a quasiregular extension $g$ of $B$ so that $\{z: g_{\overline{z}}(z)\not=0\}$ is sufficiently small (in a suitable sense) so as to imply $\phi(z)\approx z$ and hence  $g\circ\phi^{-1}=B\circ\phi^{-1}\approx B \approx f$ on $K$. We remark that our techniques build on the quasiconformal folding methods of the first author \cite{Bis15}. 

This approach yields not only information on the critical points and values of the approximants as in Theorems \ref{main_theorem} and \ref{rational_runge}, but more broadly a detailed description of the geometric structure of these approximants. We end the introduction by describing this geometric structure in a few cases. First we  introduce some more notation.

\begin{notation}\label{RMT_notation} Let $V\subset\Chat$ be a simply connected domain so that $\infty\not\in\partial V$. We let $\psi_V: \mathbf{E} \rightarrow V$ denote a Riemann mapping, where $\mathbf{E}=\mathbb{D}$ if $V$ is bounded and $\mathbf{E}=\mathbb{D}^*:=\Chat\setminus\overline{\mathbb{D}}$ if $V$ is unbounded, in which case we specify $\psi_V(\infty)=\infty$.
\end{notation}

First consider the case when $K$ is full and connected,
and $f$ is holomorphic in a neighborhood of $K$ satisfying
$||f||_K<1$. Let $\Omega$, $\Omega'$ be analytic Jordan domains
containing $K$, $f(K)$, respectively, so that $f$ is 
holomorphic in $\Omega$. Then the mapping 
\begin{equation}\nonumber
F:=\psi_{\Omega'}^{-1} \circ f \circ \psi_\Omega:\disk \to \disk
\end{equation}
is holomorphic, and by the aforementioned theorem of Carath{\'e}odory, there 
is a finite Blaschke product $b: \mathbb{D} \rightarrow \mathbb{D}$ 
that approximates $F$ on the compact set $\psi_\Omega^{-1}(K)$. Therefore 
\begin{equation} \nonumber 
B:= \psi_{\Omega'}\circ b \circ \psi_\Omega^{-1}: \Omega \rightarrow \Omega'
\end{equation}
is a holomorphic function that approximates $f$ on $K$, and
moreover $B$ restricts to an analytic, finite-to-$1$ map of
$\Gamma:=\partial\Omega$ onto $\Gamma':=\partial\Omega'$.

\begin{figure}
\centering
\scalebox{.3}{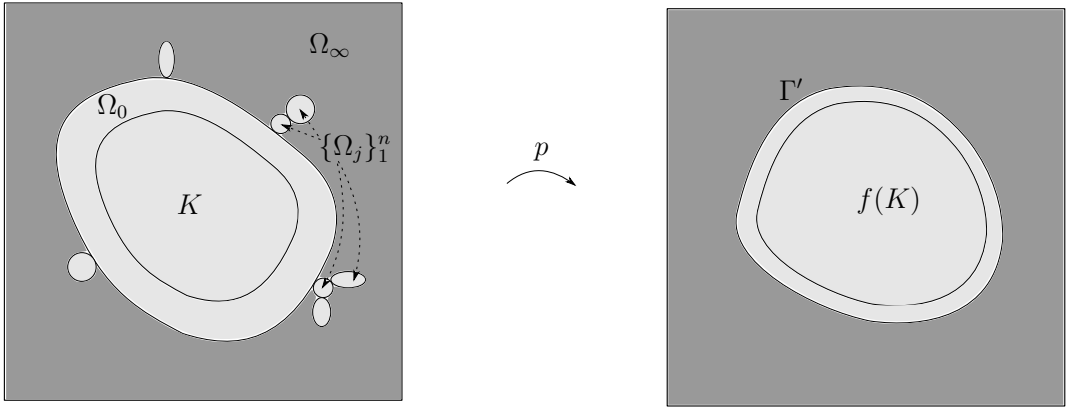}
\captionsetup{width=.9\textwidth}
\caption{This figure illustrates the geometry of the approximant $p$ in Theorem \ref{main_theorem} when $K$ is connected. The notation is explained in the text. Both domain and co-domain are colored so that regions with the same color correspond to one another under $p$. }
\label{ModelPoly}
\end{figure}

In this paper, we will show that $B$ can be approximated on $\Omega$ by a polynomial $p$ so that $p^{-1}(\Gamma')$ is an approximation of $\Gamma$. More precisely, $p^{-1}(\Gamma')$ is connected, and consists of a finite union of Jordan curves $\{ \gamma_j\}_0^n$ bounding pairwise disjoint Jordan domains $\{\Omega_j\}_0^n$ (see Figure \ref{ModelPoly}): the $\{\Omega_j\}_0^n$ are precisely the connected components of $p^{-1}(\Omega')$. There is one ``large'' component $\Omega_0$ that approximates $\Omega$ in the Hausdorff metric. The other components $\{\Omega_j\}_1^n$ can be made as small as we wish and to lie in any given neighborhood of $\partial K$. Moreover, the collection $\{\Omega_j\}_0^n$ forms a tree structure with any two boundaries $\partial\Omega_j$, $\partial\Omega_k$ either disjoint or intersecting at a single point, and with $\Omega_0$ as the ``root'' of the tree as in Figure \ref{ModelPoly}. Let $\Omega_\infty$ denote the unbounded component of $\mathbb{C}\setminus p^{-1}(\Gamma')$, so that 
\begin{equation}\label{decomp} \mathbb{C}\setminus p^{-1}(\Gamma')=\Omega_0\sqcup\left(\sqcup_{j=1}^n\Omega_j\right)\sqcup\Omega_\infty.\end{equation}
Recalling Notation \ref{RMT_notation}, the polynomial $p$ has the following simple structure with respect to the domains in (\ref{decomp}).
\begin{enumerate} 
\item $p(\Omega_0)=\Omega'$ and $\psi_{\Omega'}^{-1}\circ p\circ\psi_{\Omega_0}$ is a finite Blaschke product.
\item  $p(\Omega_j)=\Omega'$ and $p$ is conformal on $\Omega_j$ for $1\leq j \leq n$.
\item  $p(\Omega_\infty)=\mathbb{C}\setminus\overline{\Omega'}$ and $p=\psi_{\mathbb{C}\setminus\overline{\Omega'}}\circ(z\mapsto z^m)\circ\psi_{\Omega_\infty}^{-1}$ on $\Omega_\infty$ for $m=\textrm{deg}(p|_{\Omega_0})+n$.
\end{enumerate} 
In other words, up to conformal changes of coordinates, $p$ is simply a Blaschke product in $\Omega_0$, a conformal map in each $\Omega_j$, $1\leq j \leq n$, and a power map $z\mapsto z^m$ in $\Omega_\infty$. The only finite critical points of $p$ are either in $\Omega_0$, or at a point where two of the curves $(\gamma_j)_{j=1}^n$ intersect, in which case the corresponding critical value lies on $\partial\Omega'$.

Next suppose $K$ is connected, but $\mathbb{C}\setminus K$
has more than one component. In this case, in order to prove
Theorem \ref{rational_runge}, we will need to let $\Omega$ be
a multiply connected analytic domain containing $K$, and 
$\Omega'$ an analytic Jordan domain containing $f(K)$. By Grunsky's Theorem, there exists a 
proper map $b$ approximating 
$\psi_{\Omega'}^{-1}\circ f$ on $K$, so that $B:=\psi_{\Omega'}\circ b$
is a holomorphic map approximating $f$ on $K$, and $B$ restricts
to an analytic, finite-to-$1$ mapping of each component of 
$\Gamma:=\partial\Omega$ onto $\Gamma'$.
We will show $B$ can be approximated on $\Omega$ by a rational 
map $r$ so that each component of $\partial\Omega$ can be 
approximated by a component of $r^{-1}(\Gamma')$. These components of $r^{-1}(\Gamma')$
bound Jordan domains which form a decomposition of the plane as
in the previously described polynomial setting, and in the interior
of each such domain again $r$ behaves either as a proper mapping, a conformal mapping, or a power mapping
(up to conformal changes of coordinates). 

Lastly, the case when $K$ has more than one connected component
is more intricate, and we will leave the precise description to
later in the paper. 
(Briefly, quasiconformal folding is applied not just along 
the boundary of a neighborhood of $K$, but also  along 
specially chosen curves that 
connect different connected components of this neighborhood.)

We conclude the introduction by mentioning several related works. The location of $\textrm{CP}(p)$ in relation to the zeros of a polynomial $p$ is studied in the recent works \cite{MR3556367}, \cite{MR4423276}, and in \cite{MR2051397} the problem of approximation in $\mathbb{C}^n$ with prescribed critical points was studied.


\vspace{2mm}

\noindent \emph{Acknowledgements.} The authors would like to thank Dmitry Khavinson for pointing out the reference \cite{MR0463413} which simplified several arguments in a previous version of this manuscript. The authors would also like to thank Jack Burkart, Franc Forstneri\v{c}, John Garnett, Oleg Ivrii, Xavier Jarque, and Malik Younsi for their comments on this manuscript.

\section{Approximation by Proper Mappings}\label{blaschke_approximation2}

\begin{definition}\label{analytic_domain_definition} We call a domain $D\subset\mathbb{C}$ an \emph{analytic domain} if $D$ is finitely connected, and each component of $\partial D$ is an analytic Jordan curve.
 \end{definition}
 
 \noindent We remark that a boundary component of an analytic domain $D$ cannot be a single point.
 
\begin{definition}\label{proper_defn} Let $D$ be an analytic domain. We will call a continuous mapping $f:D \rightarrow \mathbb{C}$ \emph{proper} if $f(D)\subset\mathbb{D}$ and if for every compact $K\subset\mathbb{D}$, $f^{-1}(K)$ is a compact subset of $D$. 
\end{definition}

\begin{rem} Since $D$ in Definition \ref{proper_defn} is assumed to be an analytic domain, the map $f: D \rightarrow \mathbb{D}$ extends continuously to a map $f: \overline{D}\rightarrow\overline{\mathbb{D}}$, and it is straightforward to check that the map $f: D \rightarrow \mathbb{D}$ is proper if and only if $f(\partial D)\subset\mathbb{T}$.
\end{rem}

\noindent The following is Lemma 4.5.4 of \cite{MR0463413}: 

\begin{thm}\label{Grunsky_thm} Let $\varepsilon>0$, $D$ an analytic domain, $K\subset D$ compact, and $f: D \rightarrow\mathbb{D}$ holomorphic. Then there exists a proper map $B: D \rightarrow \mathbb{D}$ so that $||f-B||_K<\varepsilon$.
\end{thm}

\begin{rem} When $D=\mathbb{D}$, the function $B$ in the conclusion of Theorem \ref{Grunsky_thm} is a Blaschke product (hence the notation $B$), and in this case Theorem \ref{Grunsky_thm} is due to Carath\'eodory (see, for example, in Theorem I.2.1 of \cite{MR628971} or Theorem 5.1 of the survey \cite{MR3753897}).
\end{rem}

\begin{notation}\label{I_B} For a proper map $B$ on an analytic domain $D$, we let $\mathcal{I}_B$ denote the connected components of $\partial D\setminus\{ z  : B(z)\in\mathbb{R}\}$. In other words, $\mathcal{I}_B$ are the preimages (under $B$) of the open upper and lower half-circles $\mathbb{T}\cap\mathbb{H}$, $\mathbb{T}\cap(-\mathbb{H})$. We will frequently be dealing with sequences of proper maps $(B_n)_{n=1}^\infty$ on $D$, in which case we abbreviate $\mathcal{I}_{B_n}$ by $\mathcal{I}_n$. 
\end{notation}

In order to prove Theorems \ref{main_theorem} and \ref{rational_runge}, we will need to approximate a given function by a sequence of proper approximants $(B_n)_{n=1}^\infty$ of increasing degree, so that $|B_n'|$ is uniformly comparable to $n$ on the boundary (see (\ref{first_car_conc}) below). This will be done by post-composing the approximant coming from Theorem \ref{Grunsky_thm} with the following Blaschke products:

\begin{definition}\label{Blaschke_family} For $n\in\mathbb{N}$, $0<\delta<1$, and $0<c <1$, we define the Blaschke product: 
\begin{equation}\label{blaschkeclosetoiddefn}\mathcal{B}_{n, \delta, c}(z):=z\cdot\prod_{j=0}^{n-1}\frac{e^{2\pi ij/n}z+r}{1+re^{2\pi ij/n}z}\textrm{, where } r:=1-\frac{\delta(1-c)}{n}. 
\end{equation}
\end{definition}

\begin{prop}\label{Blaschke_props2}
If $\delta>0$ is sufficiently small, then for any $0<c<1$ we have:
\begin{align}\label{blaschke_close_to_id} \sup_{n\in\mathbb{N}}  \sup_{|z|<c}\left| \mathcal{B}_{n, \delta, c}(z)-z \right| < 4\delta. \end{align} 
\end{prop}

\begin{prop}\label{Blaschke_props} There exists a constant $C$ depending on $\delta$ and $c$, but not on $n$, so that
 \begin{align}\label{not_too_much_oscillation}  n/C < |\mathcal{B}_{n, \delta, c}'(z)| < nC \textrm{ for all } z\in\mathbb{T}\textrm{ and } n\in\mathbb{N}. \end{align} 
\end{prop}

The proofs of Proposition \ref{Blaschke_props2} and \ref{Blaschke_props} are straightforward but tedious calculations, and so we delay them until the end of the section. For now, we show how Theorem \ref{Grunsky_thm} together with Propositions \ref{Blaschke_props2} and \ref{Blaschke_props} can be used to deduce the following result, Theorem \ref{Caratheodory_for_application}. We remark that Theorem \ref{Caratheodory_for_application} is the only result from Section \ref{blaschke_approximation2} which will be needed in the remainder of the paper.

\begin{thm}\label{Caratheodory_for_application} Let $\varepsilon>0$, $D$ an analytic domain, $K\subset D$ compact, and $f: D \rightarrow\mathbb{D}$ holomorphic. Then there exists $M<\infty$ and a sequence of proper maps $(B_n)_{n=1}^\infty$ on $D$ satisfying: 
\begin{equation}\label{uniform_convergence} \sup_{n\in\mathbb{N}}||B_n -f||_{K}<\varepsilon\textrm{, and }\end{equation}
\begin{equation}\label{first_car_conc} n/M\leq|B_n'(z)|\leq nM \textrm{ for all } z\in\partial D \textrm{ and } n\in\mathbb{N}. \end{equation}

\end{thm}

\begin{proof} Fix $\varepsilon$, $D$, $K$, $f$ as in the statement of the theorem. By Theorem \ref{Grunsky_thm}, there exists a proper map $B: D \rightarrow \mathbb{D}$ satisfying $||f-B||_K<\varepsilon/2$. Fix $c$ so that $\sup_{z\in K}|B(z)|<c<1$, and fix $\delta<\varepsilon/8$ sufficiently small so that (\ref{blaschke_close_to_id}) holds. This defines the sequence $(\mathcal{B}_{n, \delta, c})_{n=1}^\infty$. We set
\begin{equation}\nonumber B_n:=\mathcal{B}_{n, \delta, c}\circ B. 
\end{equation}
The relation (\ref{uniform_convergence}) follows from (\ref{blaschke_close_to_id}) and the triangle inequality, and the relation (\ref{first_car_conc}) follows from (\ref{not_too_much_oscillation}) and the chain rule. 
\end{proof}

\noindent We conclude this section with the proofs of Propositions \ref{Blaschke_props2} and \ref{Blaschke_props}.

\vspace{2mm}

\noindent\emph{Proof of Proposition \ref{Blaschke_props2}.} Let
\begin{equation}\label{defnofB} \mathcal{B}(z):=\prod_{j=0}^{n-1}\frac{-\overline{a_j}}{|a_j|}\frac{z-a_j}{1-\overline{a_j}z}\textrm{, where } a_j:=-r\exp(-2\pi ij/n),
\end{equation}
so that $\mathcal{B}_{n, \delta, c}(z)=z\cdot\mathcal{B}(z)$. Recalling the definition of $r$ from (\ref{blaschkeclosetoiddefn}), we note that:
\begin{equation}\nonumber \sup_{|z|\leq c}\left| \frac{z+r}{1+rz} - 1\right| = \left| \frac{-c+r}{1-cr} - 1\right|=\frac{\delta(1-c)(1+c)/n}{1-c\left(1-\delta(1-c)/n\right)}\leq \frac{\delta(1+c)}{n}\leq\frac{2\delta}{n}
\end{equation}
Symmetry then gives us that:
\begin{equation}\nonumber \sup_{|z|\leq c}\left| \frac{-\overline{a_j}}{|a_j|}\frac{z-a_j}{1-\overline{a_j}z} - 1\right| \leq \frac{2\delta}{n}\textrm{ for all } 0\leq j\leq n. 
\end{equation}
Hence 
\begin{equation}\nonumber \left|\sum_{j=0}^{n-1}\log\left(  \frac{-\overline{a_j}}{|a_j|}\frac{z-a_j}{1-\overline{a_j}z} \right) \right|\leq \sum_{j=0}^{n-1}\left|\log\left(  \frac{-\overline{a_j}}{|a_j|}\frac{z-a_j}{1-\overline{a_j}z} \right) \right|\leq 3\delta
\end{equation} 
for sufficiently small $\delta$ and all $0<c<1$. Thus we conclude that 
\begin{equation}\nonumber |\mathcal{B}(z)-1|=\left| \exp\left(\sum_{j=0}^{n-1} \log\left(  \frac{-\overline{a_j}}{|a_j|}\frac{z-a_j}{1-\overline{a_j}z} \right)\right) - 1 \right| \leq 4\delta \end{equation} 
for sufficiently small $\delta$ and all $0<c<1$, and this proves (\ref{blaschke_close_to_id}) since $\mathcal{B}_{n, \delta, c}(z)=z\cdot\mathcal{B}(z)$.
\qed

\vspace{4mm}

\noindent\emph{Proof of Proposition \ref{Blaschke_props}.} We will use the notation $\lesssim$, $\simeq$, $\gtrsim$ to mean $\leq$, $=$, $\geq$ (respectively) up to a constant depending on $\delta$ and $c$, but not $n$.
Let $\mathcal{B}(z)$ be as in (\ref{defnofB}). A calculation gives:
\begin{equation}\nonumber \frac{\mathcal{B}'(z)}{\mathcal{B}(z)}=\sum_{j=0}^{n-1}\frac{1-|a_j|^2}{(1-\overline{a_j}z)^2}\cdot\left(\frac{z-a_j}{1-\overline{a_j}z}\right)^{-1}
\end{equation}
Thus we have:
\begin{equation}\label{originalBest} |\mathcal{B}'(z)| \leq \sum_{j=0}^{n-1}P(z,a_j)\textrm{, where } P(z, \zeta):=\frac{1-|\zeta|^2}{|z-\zeta|^2} \textrm{ for } z\in\mathbb{T}\textrm{, } \zeta\in\mathbb{D}
\end{equation}
 is the Poisson kernel for the unit disc.
 The sum $\sum_{j=0}^{n-1}P(z,a_j)$ takes its maximum at $z=-1$ (as well as at any of the other $n$ points $(\exp(\pi i+2\pi ij/n))_{j=0}^{n-1}$), where we have 
 \begin{equation}\label{nextBests} P(-1,r)\simeq n \textrm{ and } P(-1, a_j)\simeq n/j^2 \textrm{ for } j=1, ..., \lfloor (n-1)/2 \rfloor. 
 \end{equation}
Combining (\ref{originalBest}) and (\ref{nextBests}) together with the fact that $\sum_{j=1}^\infty1/j^2<\infty$, we have $|\mathcal{B}'(z)| \lesssim n$.
 Let us now prove $|\mathcal{B}'(z)| \gtrsim n$. The function $|\mathcal{B}'(z)|$ takes its minimum at $z_n:=\exp(\pi i(1+1/n))$, where the triangle inequality yields 
 \begin{equation}\label{RHSlabel}|\mathcal{B}'(z_n)| \geq \left|\sum_{j=0, 1}\frac{1-|a_j|^2}{(1-\overline{a_j}z_n)^2}\cdot\left(\frac{z_n-a_j}{1-\overline{a_j}z_n}\right)^{-1}\right|- \sum_{j=2}^{n-1}P(z_n,a_j) \gtrsim n. 
\end{equation}
Since $\mathcal{B}_{n, \delta, c}(z)=z\cdot\mathcal{B}(z)$ and we have shown $|\mathcal{B}'(z)|\simeq n$ for all $z\in\mathbb{T}$, The relation (\ref{not_too_much_oscillation}) follows. \qed

\section{Applying Theorem \ref{Caratheodory_for_application}}\label{blaschke_approximation}

In this brief section, we apply Theorem \ref{Caratheodory_for_application} to the setting needed for the proofs of Theorem \ref{main_theorem} and \ref{rational_runge}. Given a Jordan curve $\gamma\subset\mathbb{C}$, we denote the bounded component of $\Chat\setminus\gamma$ by $\textrm{int}(\gamma)$.

\begin{notation}\label{Blaschke_notation} We refer to Figure \ref{fig:DefnBn} for a summary of the following. For the remainder of this section, we will fix a compact set $K$, an analytic domain $D$ containing $K$, and a function $f$ holomorphic in a neighborhood of $\overline{D}$ satisfying $||f||_{D}<1$. Fix $\varepsilon>0$. We assume that \begin{equation}\label{closetoboundaryassumption} d(z, K)<\varepsilon/2 \textrm{ and } d(f(z), f(K))<\varepsilon/2 \textrm{  for every } z\in \partial D. \end{equation}

\end{notation}

\begin{figure} 
\includegraphics[scale = .35]{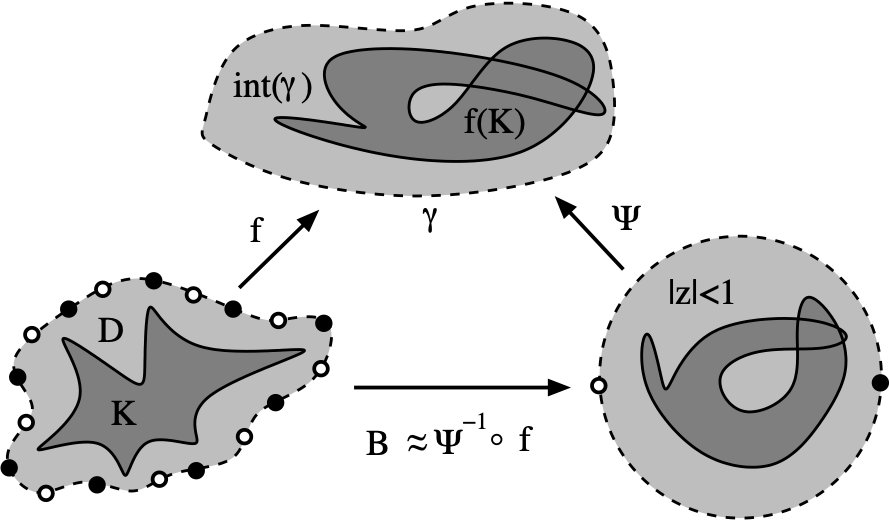}
\caption{This figure illustrates Notations \ref{I_B}, \ref{Blaschke_notation} and Theorem \ref{Caratheodory_for_application}. The vertices pictured on $\partial D$ are $B^{-1}(\pm1)$, and the components $\mathcal{I}_B$ of Notation \ref{I_B} are the edges along $\partial D$ connecting these vertices.}
 \label{fig:DefnBn}
\end{figure}


\begin{definition}\label{Psi_defn} We let $\gamma$ be an analytic Jordan curve surrounding $f(D)$ such that \begin{equation}\label{gamma_close_to_f(K)} \textrm{dist}(w, f(K))<\varepsilon \textrm{ for every } w\in\gamma, \end{equation} and let $\Psi: \mathbb{D} \rightarrow \textrm{int}(\gamma)$ denote a Riemann mapping.
\end{definition}

\noindent Recall Definition \ref{proper_defn} of a proper mapping on a domain $D$.

\begin{prop}\label{first_def_B_n} There exists $M<\infty$ and a sequence of proper mappings $(B_n)_{n=1}^\infty$ on $D$ satisfying (\ref{first_car_conc}) so that:
\begin{equation}\label{changeofcoords} ||\Psi\circ B_n-f||<\varepsilon\textrm{ for all } n\in\mathbb{N}
\end{equation}
\end{prop}

\begin{proof} Let $K'\subset D$ be compact so that $K\subset\textrm{int}(K')$. Fix $r<1$ and $\delta>0$ so that 
\begin{equation}\nonumber \Psi^{-1}\circ f(K')\subset\{ z : |z| < r \} \textrm{ and } \delta<\dist\left(\psi^{-1}\circ f(K), \partial (\psi^{-1}\circ f(K'))\right) 
\end{equation}
Apply Theorem \ref{Caratheodory_for_application} to: 
\begin{equation}\label{choices} \min\left(\frac{\varepsilon}{\sup_{|\zeta|\leq r}|\Psi'(\zeta)| }, \varepsilon, \delta \right), D, K', \Psi^{-1}\circ f. \end{equation} 
This produces a sequence of proper mappings $(B_n)_{n=1}^\infty$ on $D$ satisfying (\ref{first_car_conc}) for some $M<\infty$. We claim $(B_n)_{n=1}^\infty$ also satisfies (\ref{changeofcoords}). Indeed, for any $z\in K$ we have by our choice of $\delta$ that $|B_n(z)|\leq r$ and $|\Psi^{-1}\circ f(z)|\leq r$. Thus we deduce:
\begin{equation}\nonumber |\Psi\circ B_n(z)-f(z)| = |\Psi\circ B_n(z)-\Psi\circ\Psi^{-1}\circ f(z)| \leq \sup_{|\zeta|\leq r}|\Psi'(\zeta)|\cdot |B_n(z)-\Psi^{-1}\circ f(z)|,
\end{equation}
and so (\ref{changeofcoords}) follows from (\ref{choices}).
\end{proof}



Recall from the introduction that we plan to extend the definition of the approximant $\Psi\circ B_n\approx f$ from $D$ to all of $\mathbb{C}$.  To this end, it will be useful to define the following graph structure on $\partial D$.

\begin{definition}\label{vertex_disc_definition} For any $n\in\mathbb{N}$, we define a set of vertices on $\partial D$ by $\mathcal{V}_n:=(B_n|_{\partial D})^{-1}(\mathbb{R})$, where each vertex $v$ is labeled black or white according to whether $B_n(v)>0$ or $B_n(v)<0$, respectively. The curve $\partial D$ will be considered as a graph with edges defined by $\mathcal{I}_{n}$ (recall from Notation \ref{I_B} that $\mathcal{I}_{n}$ is precisely the collection of components of $\partial D \setminus\mathcal{V}_n$). We will sometimes write $D_n$ in place of $D$ when we wish to emphasize the dependence of the graph $\partial D$ on $n$. 
\end{definition}






\begin{definition}\label{quasiregular_disc_definition} We define a holomorphic mapping $g_n$ in $D$ by the formula \begin{equation}\label{g_n_disc_defn} g_n(z):= \Psi \circ B_n(z). \end{equation}
\end{definition}

In Sections \ref{folding_section}-\ref{quasiregular_approx} we will quasiregularly extend the definition of $g_n$ to $\mathbb{C}$, and then in Section \ref{main_proofs_section} we apply the MRMT to produce the rational approximant of Theorem \ref{rational_runge} as described in the introduction.

\begin{rem}\label{dependence_notation} Recall that in Notation \ref{Blaschke_notation}, we fixed $\varepsilon>0$, a compact set $K$ contained in an analytic domain $D$, and a function $f$ holomorphic in $D$ (we note $\varepsilon$, $K$, $D$, $f$ also satisfied extra conditions specified in Notation \ref{Blaschke_notation}). The objects $\gamma$, $\Psi$, $B_n$, $\mathcal{V}_n$, $g_n$ we then defined in this section were determined by our initial choice of $\varepsilon$, $K$, $D$, $f$. In future sections, it will be useful to think of  $\gamma$, $\Psi$, $B_n$, $\mathcal{V}_n$, $g_n$ as defining functions which take as input some quadruple $(\varepsilon, K, D, f)$ (for any $\varepsilon$, $K$, $D$, $f$ as in Notation \ref{Blaschke_notation}), and output whatever object we defined in this section. For instance, $\mathcal{V}_n$ defines a function which takes as input any $(\varepsilon, K, D, f)$ as in Notation \ref{Blaschke_notation} and outputs (via Definition \ref{vertex_disc_definition}) a set of vertices $\mathcal{V}_n(\varepsilon, K, D, f)$ on $\partial D$. Similarly, $B_n$ takes as input any $(\varepsilon, K, D, f)$ as in Notation \ref{Blaschke_notation} and outputs (via Proposition \ref{first_def_B_n}) a proper mapping $B_n(\varepsilon, K, D, f)$ on $D$. Likewise for  $\gamma$, $\Psi$, $g_n$.


\end{rem}

\section{Quasiconformal Folding}\label{folding_section}

Given a compact set $K\subset\mathbb{C}$ and a function $f$ holomorphic in a domain $D$ containing $K$, we showed in Section \ref{blaschke_approximation} how to approximate $f$ by a holomorphic function $g_n$ defined in $D$ (see Definition \ref{quasiregular_disc_definition}). If $f$ is a function holomorphic in an arbitrary analytic neighborhood $U$ (where $U$ need not be connected) of a compact set $K$, then one can apply the results of Section \ref{blaschke_approximation} to each component of $U$ which intersects $K$ (this is done precisely in Definition \ref{many_comps_vertex_defn}): this yields a holomorphic approximant of $f$ defined in a finite union of domains. In Sections \ref{folding_section}-\ref{extension_section4}, we will build the apparatus necessary to extend this holomorphic approximant to a quasiregular function of $\mathbb{C}$ which is holomorphic outside a small set.




It was convenient to assume in Notation \ref{Blaschke_notation} that the compact set $K$ was covered by a single domain $D$, however we now begin to work more generally:

\begin{rem}\label{notation_Section4} We refer to Figure \ref{fig:notation_Section4} for a summary of the following. Throughout Sections \ref{folding_section}-\ref{extension_section4}, we will fix $\varepsilon>0$, a compact set $K\subset \mathbb{C}$, a domain $\mathcal{D}$ containing $K$, a disjoint collection of analytic domains $(D_i)_{i=1}^k$ such that $K\subset U:=\cup_i D_i\subset \mathcal{D}$, and a function $f$ holomorphic in a neighborhood of $\overline{U}$ satisfying $||f||_{\overline{U}}<1$. We assume that the following analog of Equation (\ref{closetoboundaryassumption}) holds:
\begin{equation}\nonumber d(z, K\cap D_i) < \varepsilon/2 \text{ and  } d(f(z) , f(K\cap D_i)) < \varepsilon/2 \textrm{ for all } z \in \partial D_i \textrm{ and } 1\leq i \leq k. \end{equation}
\end{rem}

Applying the methods of the previous section to each component $D_i$ of $U$, we can define a sequence of proper mappings  $(B_n)_{n=1}^\infty$ on each $D_i$ (see Remark \ref{dependence_notation}). We will let $B_n$ denote the corresponding function defined on $U$. In particular, $(B_n)_{n=1}^\infty$ gives the following definition of vertices on the boundary of $U=\cup_i D_i$ (see Definition \ref{vertex_disc_definition} and Remark \ref{dependence_notation}).

\begin{definition}\label{many_comps_vertex_defn_init} For every $n\in\mathbb{N}$, we define a set of vertices $\mathcal{V}_n$ on $\partial U$ by \begin{equation}\nonumber \mathcal{V}_n := \bigcup_{i=1}^k \mathcal{V}_n(\varepsilon, K\cap D_i, D_i, f|_{D_i}) = \bigcup_{i=1}^k (B_n|_{\partial D_i})^{-1}(\reals). \end{equation}
\end{definition}

We now extend the graph structure on $\partial U$ by connecting the different components of $U$ by curves $\{\Gamma_i\}_{i=1}^{k-1}$ in Proposition \ref{existence_of_curves} below, and defining vertices along these curves in Definition \ref{curve_vertices_defn}. We will need to prove a certain level of regularity for these curves and vertices in order to ensure that the dilatations of quasiconformal adjustments we will make later do not degenerate as $n\rightarrow\infty$. We will denote the curves by $\{\Gamma_i\}_{i=1}^{k-1}$, and we remark that the curves depend on $n$, although we suppress this from the notation.




\begin{prop}\label{existence_of_curves} For each $n\in\mathbb{N}$, there exists a collection of disjoint, closed, analytic Jordan arcs $\{\Gamma_i\}_{i=1}^{k-1}$ in $(\Chat\setminus U)\cap \mathcal{D}$ satisfying the following properties: 
\begin{enumerate} 
\item Each endpoint of $\Gamma_i$ is a vertex in $\mathcal{V}_n$, 
\item Each $\Gamma_i$ meets $\partial U$ at right angles, 
\item $U\cup\left(\cup_i\Gamma_i\right)$ is connected, and
\item For each $1\leq i \leq k-1$, the sequence (in $n$) of curves $\Gamma_i$ has an analytic limit.
 \end{enumerate}
\end{prop}

\begin{proof}  The set $(\Chat\setminus U)\cap \mathcal{D}$ must contain at least one simply-connected region $V$ with the property that there are distinct $i$, $j$ with both $\partial V\cap \partial D_i$ and $\partial V\cap \partial D_j$ containing non-trivial arcs (see Figure \ref{fig:DefnV.eps}). By (\ref{first_car_conc}), for all sufficiently large $n$ both $\partial V\cap \partial D_i$, $\partial V\cap \partial D_{j}$ contain vertices of $\mathcal{V}_n$ which we denote by $v_i\in\partial D_i$, $v_{j}\in\partial D_{j}$, respectively. Consider a conformal map $\phi: \mathbb{D} \rightarrow V$, and define $\Gamma_1$ to be the image under $\phi$ of the hyperbolic geodesic connecting $\phi^{-1}(v_i)$, $\phi_i^{-1}(v_{j})$ in $\mathbb{D}$. 

\begin{figure} 
\includegraphics[scale = .25]{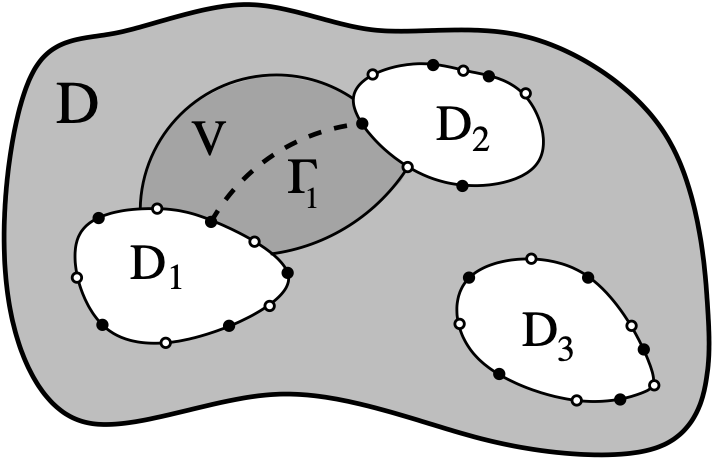}
\caption{Illustrated is the Definition of the domain $V$ in the proof of Proposition \ref{existence_of_curves}}
 \label{fig:DefnV.eps}
\end{figure}

We now proceed recursively, making sure at step $l$ we pick a $V$ which connects two components of $U$ not already connected by a $\Gamma_1$, ..., $\Gamma_{l-1}$, and so that $V$ is disjoint from $\Gamma_1$, ..., $\Gamma_{l-1}$. The curves $\Gamma_i$ satisfy conclusions (1)-(3) of the proposition. We may ensure that for each $1\leq i \leq k-1$, the sequence (in $n$) of curves $\Gamma_i$ has an analytic limit by choosing $v_i$, $v_j$ above to converge as $n\rightarrow\infty$.
\end{proof}

\begin{definition}\label{curve_vertices_defn} Consider the vertices $\mathcal{V}_n\subset \partial U$ of Definition \ref{many_comps_vertex_defn_init}. We will augment $\mathcal{V}_n$ to include vertices on the curves $(\Gamma_i)_{i=1}^{k-1}$ as follows (see Figure \ref{fig:GammaVerts.eps}).  Let $\Gamma\in(\Gamma_i)_{i=1}^{k-1}$ denote both the curve as a subset of $\mathbb{C}$ and the arclength parameterization of the curve, and suppose $\Gamma$ connects vertices \[\Gamma(0)=v_i\in \partial D_i, \hspace{2.5mm} \Gamma(\textrm{length}(\Gamma))=v_{j}\in\partial D_{j}.\] For $k=i$, $j$, let $\varepsilon_k$ denote the minimum length of the two edges with endpoint $v_k$ in $\partial D_k$, and suppose without loss of generality $\varepsilon_{j}<\varepsilon_{i}$. Let $l$ be so that \[\varepsilon_{j}/2\leq\varepsilon_i/2^l\leq2\varepsilon_{j}.\] We place vertices at $\Gamma(\varepsilon_i/2), ..., \Gamma(\varepsilon_i/2^l)$, and we place vertices along $\Gamma([\varepsilon_i/2^l, \textrm{length}(\Gamma)])$ at equidistributed points. We can label the vertices black/white along $\Gamma$ so that vertices connect only to vertices of the opposite color by adding one extra vertex at the midpoint of the segment having $v_{j}$ as an endpoint, if need be.%
\end{definition}

\begin{figure} 
\includegraphics[scale = .2]{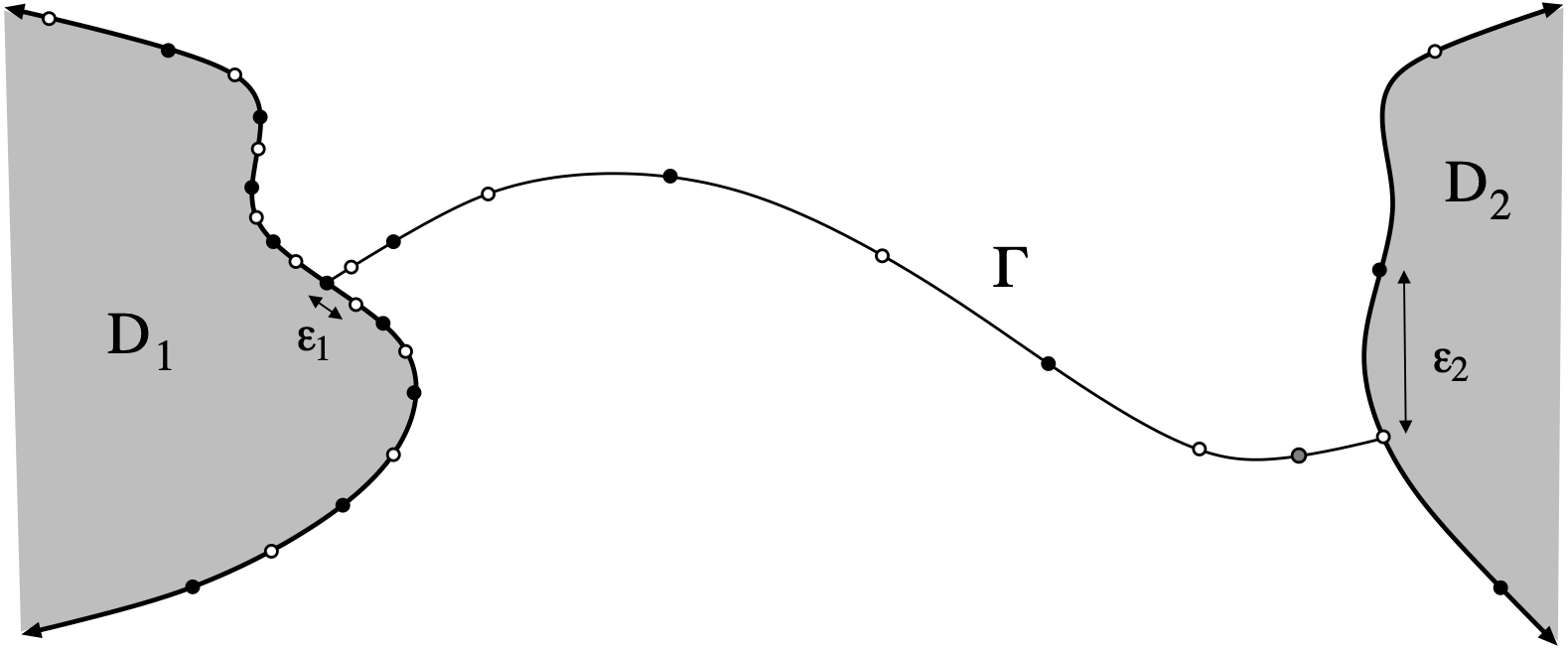}
\caption{Illustrated is Definition \ref{curve_vertices_defn}.}
 \label{fig:GammaVerts.eps}
\end{figure} 


\noindent We introduce the following notation.

\begin{notation}\label{omega_def_setup} Throughout Sections \ref{folding_section}-\ref{extension_section4}, we will let $\Omega$ denote a fixed (arbitrary) component of \begin{equation}\label{complcomp} \Chat\setminus\left( \overline{U} \cup \bigcup_{i=1}^{k-1}\Gamma_i \right), \end{equation} and $p\in\Omega$. Note that $\Omega$ is simply connected by Proposition \ref{existence_of_curves}(3). Denote $\mathbb{D}^*:=\Chat\setminus\overline{\mathbb{D}}$, and let $\sigma$ denote any conformal mapping \begin{equation}\label{covering_map_notation} \sigma: \mathbb{D}^* \rightarrow \Omega \end{equation} satisfying $\sigma(\infty)=p$. For $z\in\Omega$, we define $\tau(z):=\sigma^{-1}(z)$. The map $\tau$ induces a partition of $\mathbb{T}$ which we denote by $V_n:=\tau(\mathcal{V}_n)$.
\end{notation}

\begin{figure} 
\includegraphics[scale = .2]{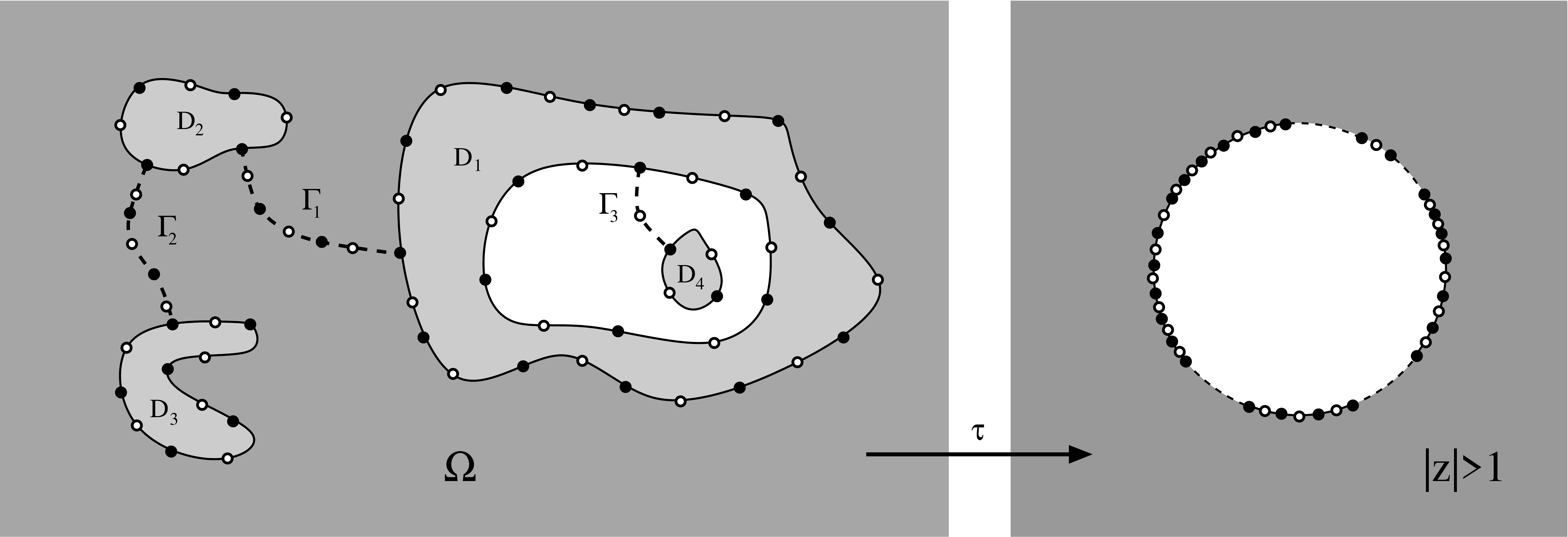}
\caption{This figure illustrates Remark \ref{notation_Section4} and Notation \ref{omega_def_setup}. As pictured, $U$ has four components $(D_i)_{i=1}^4$ which are connected by curves $(\Gamma_i)_{i=1}^3$. Recall $K\subset U$ (the compact set $K$ is not shown in the figure). The unbounded component $\Omega$ of (\ref{complcomp}) is pictured in dark grey. The map $\tau: \Omega \rightarrow \mathbb{D}^*$ is a conformal mapping, and sends the vertices on $\partial\Omega$ to (possibly unevenly spaced) vertices on the unit circle. }
 \label{fig:notation_Section4}
\end{figure}


\begin{rem} We will sometimes write $\Omega_n$, $\mathbb{D}_n^*$ in place of $\Omega$, $\mathbb{D}^*$, respectively, when we wish to emphasize the dependence of the vertices $\mathcal{V}_n\subset\partial\Omega$, $V_n\subset\partial\mathbb{D}^*$ on the parameter $n$.
\end{rem}

\begin{prop}\label{max_diam_omega} For the graph $\partial \Omega_n$, we have: \begin{equation}\nonumber \max\{\emph{diam}(e) : e \emph{ is an edge of } \partial \Omega_n \} \xrightarrow{n\rightarrow\infty}0. \end{equation} 
\end{prop}

\begin{proof} This follows from (\ref{first_car_conc}) and Definition \ref{curve_vertices_defn}.
\end{proof}

As explained in the introduction, in order to prove uniform approximation in Theorem \ref{rational_runge}, we will need to prove that our quasiregular extension is holomorphic outside a region of small area. This will usually mean proving the following condition holds.

\begin{definition}\label{vertsuppdefn} Suppose $V\subset\mathbb{C}$ is an analytic domain, and $\partial V$ is a graph. Let $C>0$. We say a quasiregular mapping $\phi: V \rightarrow \phi(V)$ is $C$-\emph{vertex-supported} if \begin{equation}\label{defvertsupp} \textrm{supp}(\phi_{\overline{z}}) \subset \bigcup_{e\in\partial V} \{ z : \dist(z,e) < C\cdot\textrm{diam}(e) \} \end{equation} 
(see Figure \ref{fig:C_nbhd}), where the union in (\ref{defvertsupp}) is taken over all edges $e$ on $\partial V$.
\end{definition}

\begin{figure} 
\includegraphics[trim = 100 0 100 0, clip, height= 1.9in]{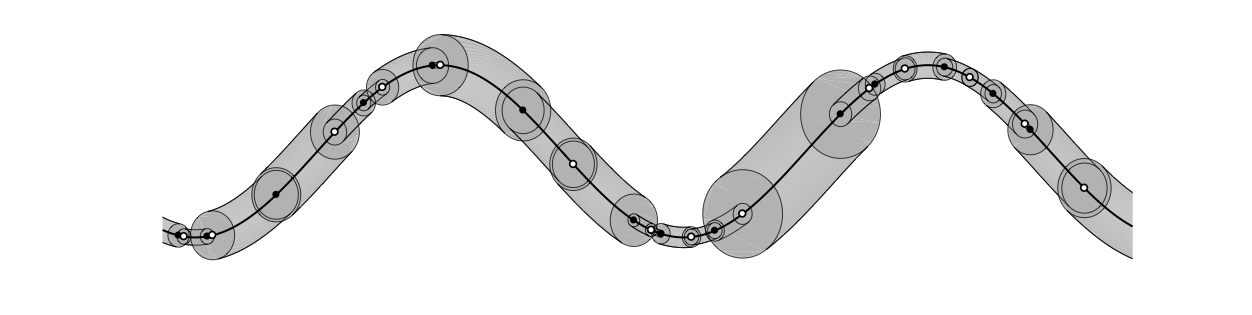}
\caption{Shown as a black curve is part of a graph $G$, and in light gray the neighborhood $\cup_{e\in G} \{ z : \dist(z,e) < C\cdot\textrm{diam}(e) \}$ of $G$. }
 \label{fig:C_nbhd}
\end{figure}

\noindent It will also be useful to have the following definition.



\begin{definition}\label{length_mult_defn} Suppose $e$, $f$ are rectifiable Jordan arcs, and $h: e \rightarrow f$ is a homeomorphism. We say that $h$ is \emph{length-multiplying} on $e$ if  the push-forward (under $h$) of arc-length measure on $e$ coincides with the arc-length measure on $f$ multiplied by $\textrm{length}(f)/\textrm{length}(e)$.
\end{definition}

First we will adjust the conformal map $\tau$ so as to be length-multiplying along edges of $\partial\Omega$. Recall the vertices $V_n\subset\mathbb{T}$ defined in Notation \ref{omega_def_setup}.

\begin{prop}\label{first_folding_adjustment} For every $n$, there is a $K$-quasiconformal mapping $\lambda: \mathbb{D}_n^* \rightarrow \mathbb{D}_n^*$ so that: \begin{enumerate} \item $\lambda$ is $C$-vertex-supported for some $C>0$, \item $\lambda(z)=z$ on $V_n$ and off of $\textrm{supp}(\lambda_{\overline{z}})$, \item $\lambda\circ\tau$ is length-multiplying on every component of $\partial\Omega\setminus \mathcal{V}_n$, \item $C$, $K$ do not depend on $n$. \end{enumerate} 
\end{prop}

\begin{proof} This is a consequence of Theorem 4.3 of \cite{Bis15}. Indeed, recall $\tau:=\sigma^{-1}$ and consider the $2\pi i$-periodic covering map
\begin{equation}\label{covering_map}\phi:=\sigma\circ\exp: \mathbb{H}_r\mapsto \Omega. \end{equation}
The map $\phi$ induces a periodic partition $\phi^{-1}(\mathcal{V}_n)$ of $\partial\mathbb{H}_r$ which has \emph{bounded geometry} (see the introduction of \cite{Bis15}, or Section 2 of \cite{MR4023391}) with constants independent of $n$ by Proposition \ref{existence_of_curves}(2) and Definition \ref{curve_vertices_defn}. Thus Theorem 4.3 of  \cite{Bis15} applies to produce a $2\pi i$-periodic, $C$ vertex-supported, and $K$-quasiconformal map $\beta: \mathbb{H}_r\rightarrow\mathbb{H}_r$ so that $\phi\circ\beta$ is length-multiplying on edges of $\mathbb{H}_r$, and $C$, $K$ are independent of $n$. Thus, the inverse
\begin{equation}\nonumber \beta^{-1}\circ\log\circ\tau \end{equation}
is length-multiplying, and since $\exp$ is length-multiplying on vertical edges, the well-defined map
\begin{equation}\nonumber \lambda:=\exp\circ\beta^{-1}\circ\log: \mathbb{D}^* \rightarrow \mathbb{D}^* \end{equation}
satisfies the conclusions of the Proposition.

\end{proof}


The main idea in defining the quasiregular extension in $\Omega$ is to send each edge of $\partial D_i$ to the upper or lower half of the unit circle by following $\lambda\circ\tau$ with a power map $z\mapsto z^n$ of appropriate degree. The main difficulty in this approach, however, is that the images of different edges of $\partial D_i$ under $\lambda\circ\tau$ may differ significantly in size, so that there is no single $n$ with $z\mapsto z^n$ achieving the desired behavior. The solution is to modify the domain $\Omega$ by removing certain ``decorations'' from the domain $\Omega$, so that each edge of $\partial D_i$ is sent to an arc of roughly the same size under $\lambda\circ\tau$. This is formalized below in Theorem \ref{complicated_folding_adjustment} (see also Figures \ref{fig:folding_picture}, \ref{fig:DefnPsi}), and is an application of the main technical result of \cite{Bis15} (see Lemma 5.1). The ``decorations'' are the trees in the following definition.

\begin{definition}\label{treedomdefn} Let $V\subset \mathbb{T}$ be a discrete set. We call a domain $W\subset\mathbb{D}^*$ a \emph{tree domain rooted at $V$} if $W$ consists of the complement in $\mathbb{D}^*$ of a collection of disjoint trees, one rooted at each vertex of $V$ (see the center of Figure \ref{fig:folding_picture}). 
\end{definition}

\begin{notation} For $m\in\mathbb{N}$, we let \begin{equation} \nonumber \mathcal{Z}_m^{\pm}:=\{ z \in \mathbb{T} : z^m=\pm1 \}, \end{equation} \begin{equation} \nonumber \mathcal{Z}_m:=\mathcal{Z}_m^+\cup\mathcal{Z}_m^-. \end{equation} In other words, $\mathcal{Z}_m^+$ denotes the $m^{\textrm{th}}$ roots of unity, and $\mathcal{Z}_m^-$ the $m^{\textrm{th}}$ roots of $-1$.
\end{notation}


\begin{figure} 
\includegraphics[scale = .17]{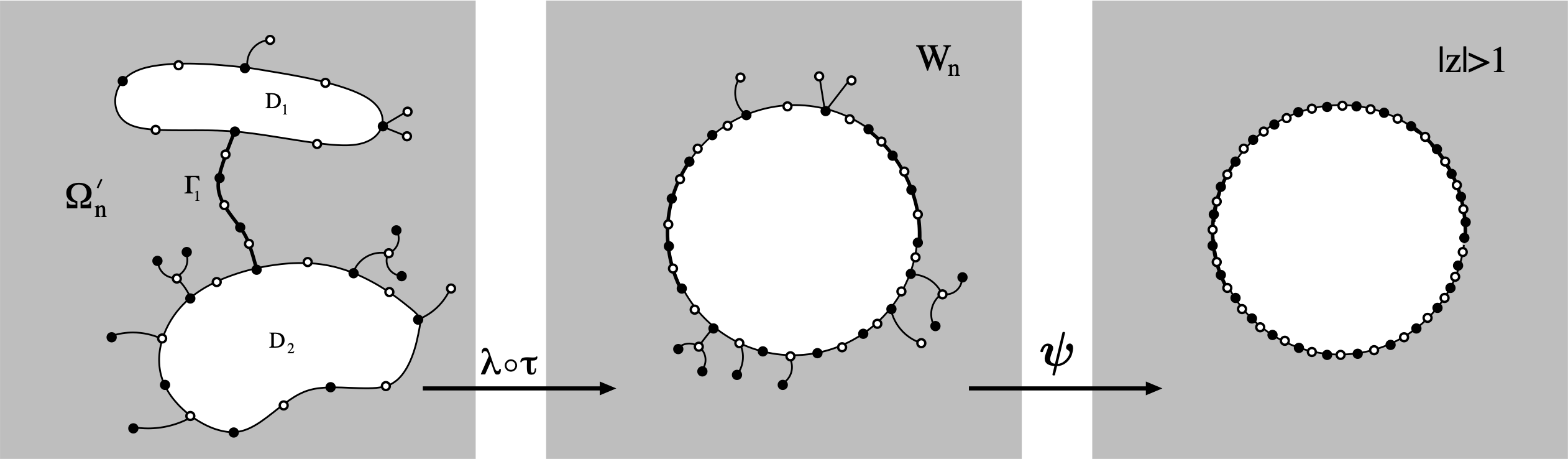}
\caption{This figure illustrates the Folding Theorem \ref{complicated_folding_adjustment} and Notation \ref{modified_domain}. The simply connected domain $\Omega_n'$ is obtained by removing from $\Omega$ certain trees based at the vertices along $\partial\Omega$.  }
 \label{fig:folding_picture}
\end{figure}

\begin{thm}\label{complicated_folding_adjustment} For every $n$, there exists a tree domain $W_n$ rooted at $V_n$, an integer $m=m(n)$, and a $K$-quasiconformal mapping $\psi: W_n \rightarrow \mathbb{D}^*$ so that: \begin{enumerate} \item $\psi$ is $C$-vertex-supported for some $C>0$, and $\psi(z)=z$ off of $\textrm{supp}(\psi_{\overline{z}})$, \item on any edge $e$ of $\partial W_n\cap \mathbb{T}$, $\psi$ is length-multiplying and $\psi(e)$ is an edge in $\mathbb{T}\setminus\mathcal{Z}_m$,  \item for any edge $e$ of $\partial W_n\cap\mathbb{D}^*$, $\psi(e)$ consists of two edges in $\mathbb{T}\setminus\mathcal{Z}_m$. Moreover, if $x\in e$, the two limits $\lim_{W_n\ni z\rightarrow x}\psi(z)\in \mathbb{T}$ are equidistant from $\mathcal{Z}_m^+$, and from $\mathcal{Z}_m^-$, and \item $C$, $K$ do not depend on $n$. \end{enumerate}
\end{thm}


\begin{proof} We consider the $2\pi$-periodic covering map
\begin{equation}\label{covering_map2}\phi:=\sigma\circ\lambda\circ\exp\circ(z\mapsto -iz): \mathbb{H}\mapsto \Omega. \end{equation}
inducing a periodic partition $\phi^{-1}(\mathcal{V}_n)$ of $\partial\mathbb{H}$. By (\ref{first_car_conc}), Definition \ref{curve_vertices_defn}, and Proposition \ref{first_folding_adjustment}(2), any two edges of $\mathbb{H}$ have comparable lengths with constant independent of $n$. Therefore, Lemma 5.1 of \cite{Bis15} applies to yield a $2\pi$-periodic $K$-quasiconformal map $\Psi_n$ of $\mathbb{H}$ onto a subdomain $\Psi_n(\mathbb{H})\subsetneq\mathbb{H}$, with $K$ independent of $n$. We let 
\begin{equation}\nonumber W_n:=\exp(-i\Psi_n^{}(\mathbb{H})) \end{equation}
and
\begin{equation}\label{psifoldingdefn} \psi:= \exp\circ-i\Psi_n^{-1}\circ i\log: W_n\rightarrow\mathbb{D}^*. \end{equation}
The map (\ref{psifoldingdefn}) is well-defined, and the conclusions of the theorem follow from Lemma 5.1 of \cite{Bis15}.
\end{proof}

\begin{notation}\label{modified_domain} We will use the notation $\Omega_n':=(\lambda\circ\tau)^{-1}(W_n)$ (see Figure \ref{fig:folding_picture}).
\end{notation}

\begin{figure} 
\includegraphics[scale = .25]{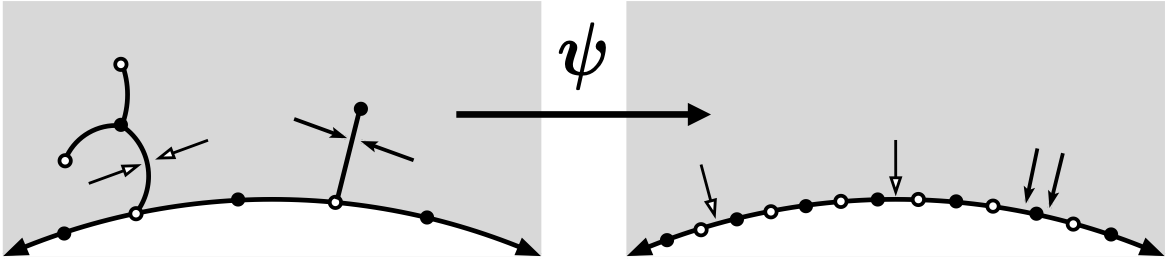}
\caption{ For any $x\in\partial W_n\cap\mathbb{D}^*$, there are two limits $\lim_{W_n\ni z\rightarrow x}\psi(z)\in \mathbb{T}$ as illustrated in this figure. Theorem \ref{complicated_folding_adjustment}(3) says that these two limits are equidistant from the nearest black vertex, and are equidistant from the nearest white vertex.  }
 \label{fig:DefnPsi}
\end{figure}


\section{Annular Interpolation Between the Identity and a Conformal Mapping}\label{extension_section}

Recall from Notation \ref{notation_Section4} that we have fixed $\varepsilon>0$, a compact set $K$, disjoint analytic domains $(D_i)_{i=1}^k$ so that $U:=\cup_iD_i$ contains $K$, and $f$ holomorphic in a neighborhood of $\overline{U}$ with $||f||_{\overline{U}}<1$. In this section, we briefly define two useful interpolations in Lemmas \ref{eta_i_psi} and \ref{squished_eta} which we will need.

Since the domain $D_i$ contains the compact set $K\cap D_i$, the definitions and results of Section \ref{blaschke_approximation} apply to $(\varepsilon, K\cap D_i, D_i, f|_{D_i})$ for each $1\leq i \leq k$ (see Notation \ref{Blaschke_notation}). Thus Remark \ref{dependence_notation} applies to define (\ref{gamma_i_defn}), (\ref{Psi_i_defn}) and (\ref{newbdefn}) in the following.


\begin{definition}\label{many_comps_vertex_defn} Let $1\leq i \leq k$. We define the Jordan curve \begin{equation}\label{gamma_i_defn}\gamma_i := \gamma(\varepsilon, K\cap D_i, D_i, f|_{D_i}). \end{equation} Recalling that $\textrm{int}(\gamma_i)$ denotes the bounded component of $\Chat\setminus\gamma_i$, we define 
\begin{equation}\label{Psi_i_defn} \Psi_i:=\Psi(\varepsilon, K\cap D_i, D_i, f|_{D_i}) \end{equation}
 to be a Riemann mapping $\Psi_i: \mathbb{D} \rightarrow \textrm{int}(\gamma_i)$. Lastly, we define the proper mappings
\begin{equation} \label{newbdefn} B_n:=B_n(\varepsilon, K\cap D_i, D_i, f|_{D_i}) \textrm{ on } D_i, \end{equation} where we suppress the dependence of $(B_n)_{n=1}^\infty$ on $i$ from the notation. 
\end{definition}

\noindent Recall that in Section \ref{folding_section}, we defined curves $\{\Gamma_i\}_{i=1}^{k-1}$ connecting the domains $D_i$, and in Notation \ref{omega_def_setup} we fixed a component $\Omega$ of the complement of $\overline{U}\cup\cup_{i=1}^{k-1}\Gamma_i$.

\begin{notation}\label{ell_defn} After relabeling the $(D_i)_{i=1}^k$ if necessary, there exists $1\leq \ell \leq k$ so that $\partial D_i \cap \partial\Omega\not=\emptyset$ if and only if $i \leq \ell$ (see Figure \ref{fig:notation_Section4} for example). For each $1\leq i \leq \ell$, note that the intersection $\partial D_i\cap \partial\Omega$ consists of a single Jordan curve which is mapped onto $\mathbb{T}$ by $B_n$.
\end{notation}


The two interpolations we will need are given in Lemmas \ref{eta_i_psi} and \ref{squished_eta} below. In Lemma \ref{eta_i_psi}, we define an interpolation $\eta_i^\Psi$ between $z\mapsto z$ on $|z|=2$ with $z\mapsto\Psi_i(z)$ on $|z|=1$ (see Figure \ref{fig:DefnEta}), and in Lemma \ref{squished_eta} we modify $\eta_i^\Psi$ to define a map $\eta_i$ so that $\eta_i(z)=\eta_i(\overline{z})$ for $|z|=1$.

\begin{figure} 
\includegraphics[scale = .23]{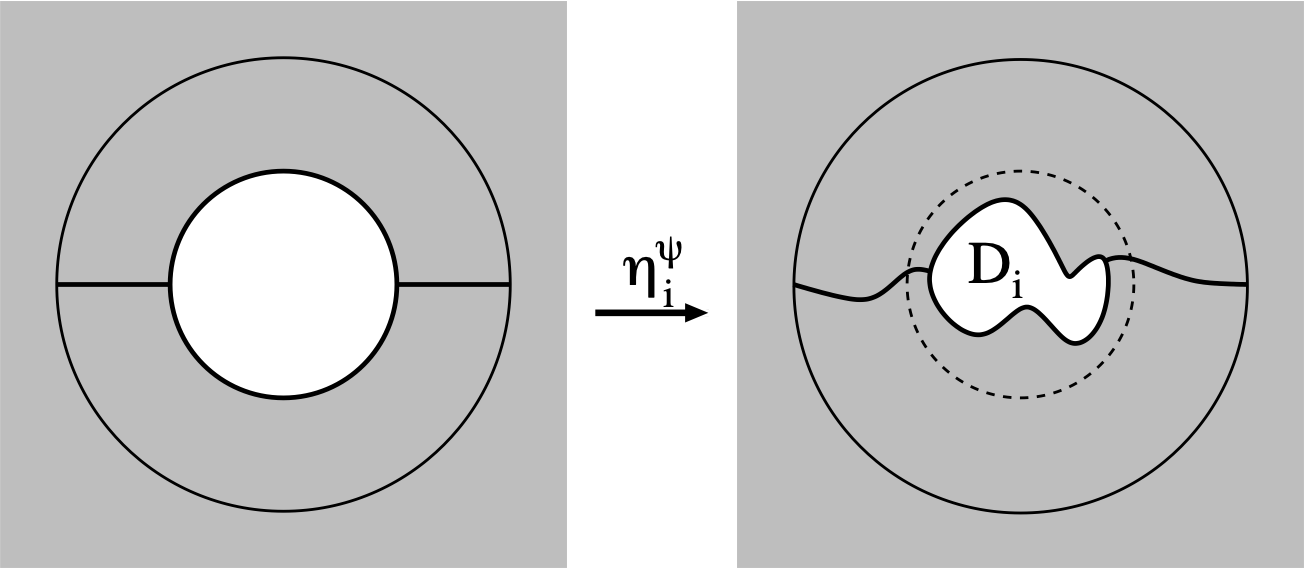}
\caption{ Illustrated is the map $\eta_i^\Psi: \mathbb{D}^* \rightarrow \Chat\setminus\Psi_i(\mathbb{D})$ of Lemma \ref{eta_i_psi}. The dotted circle on the right depicts the unit circle. }
 \label{fig:DefnEta}
\end{figure}

\begin{lem}\label{eta_i_psi} For each $1\leq i\leq\ell$, there is a quasiconformal mapping $\eta_i^\Psi: \mathbb{D}^* \rightarrow \Chat\setminus\Psi_i(\mathbb{D})$ satisfying the relations: 
\begin{equation}\label{first_desired} \eta_i^\Psi(z)=z \emph{ for } |z|\geq 2\emph{ and }  \end{equation}
\begin{equation}\label{second_desired} \eta_i^\Psi(z)=\Psi_i(z)\emph{ for all } |z|=1. \end{equation}
Moreover, if $D_i$, $D_j$ for $1\leq i, j\leq\ell$ are connected by one of the curves $(\Gamma_i)_{i=1}^{k-1}$, then
\begin{equation}\label{third_desired} \eta_i^\Psi([-2,-1])\cap\eta_{j}^\Psi([1,2])=\eta_i^\Psi([1,2])\cap\eta_{j}^\Psi([-2,-1])=\emptyset. \end{equation}
\end{lem}

\begin{proof} The existence of $\eta_i^\Psi$ satisfying (\ref{first_desired}) and  (\ref{second_desired}) follows from a standard lemma on the extension of quasisymmetric maps between boundaries of quasiannuli (see, for instance, Proposition 2.30(b) of \cite{MR3445628}). If (\ref{third_desired}) fails for the collection $(\eta^\Psi_i)_{i=1}^l$ thus defined, we can renormalize the conformal mappings $(\Psi_i)_{i=1}^\ell$ appropriately (to rotate the points $\Psi_i(\pm 1)$ along the curve $\Psi_i(\mathbb{T})$), and post-compose a subcollection of the $\eta^\Psi_i$ by diffeomorphisms of $2\mathbb{D} \setminus \Psi_i(\mathbb{D})$ so that (\ref{third_desired}) is satisfied, and (\ref{first_desired}) and  (\ref{second_desired}) still hold. 
\end{proof}

\begin{lem}\label{squished_eta} For each $1\leq i\leq\ell$, there is a quasiconformal mapping
 \begin{equation}\nonumber \eta_i: \mathbb{D}^* \rightarrow \mathbb{C}\setminus \Psi_i([-1,1])  \end{equation}
satisfying the relations
\begin{equation}\label{first_eta_i_R} \eta_i(z)=z \emph{ for } |z|\geq 2, \end{equation}
 \begin{equation}\label{conjugacy_relation} \eta_i(z)=\eta_i(\overline{z}) \emph{ for } |z|=1\emph{, and } \end{equation}
 \begin{equation}\label{last_eta_i_R} \eta_i(z)=\eta_i^\Psi(z)\emph{ for } z\in\mathbb{R}\cap\mathbb{D}^*.  \end{equation}
\end{lem}

\begin{proof} Define
\begin{equation}\nonumber \gamma_i^+ := \eta_i^\Psi(\partial(A(1,2)\cap \mathbb{H})). \end{equation} 
Let $\eta$ be a quasisymmetric mapping of $\mathbb{T}\cap\mathbb{H}$ onto $[-1,1]$ fixing $\pm1$ (one can take $\eta:=M|_{\mathbb{T}\cap\mathbb{H}}$ where $M$ is a Mobius transformation mapping $-1$, $1$, $i$ to $-1$, $1$, $0$, respectively).
Define a mapping $g$ on $\gamma_i^+$ by: 
\begin{equation}\label{definition_of_g_CB} g(z):= \begin{cases} 
 \Psi_i\circ\eta\circ\Psi_i^{-1}(z) & z\in\Psi_i(\mathbb{T}\cap\mathbb{H}) \\
 z & \textrm{otherwise} \\	 
  \end{cases} 
\end{equation}
Since $g$ is a quasisymmetric mapping, a standard lemma on extension of quasisymmetric maps between boundaries of quasidisks (see, for instance, Proposition 2.30(a) of \cite{MR3445628}) implies that $g$ may be extended to a quasiconformal mapping of $\eta_i^\Psi(A(1,2) \cap \mathbb{H})$. Define $g$ similarly in $\eta_i^\Psi(A(1,2)\cap(-\mathbb{H}))$. We let $\eta_i:=g\circ\eta_i^\Psi$. It is then straightforward to check that $\eta_i$ satisfies (\ref{first_eta_i_R})-(\ref{last_eta_i_R}).
\end{proof}

\begin{rem} Lemmas \ref{eta_i_psi} and \ref{squished_eta} define $2\ell$ many quasiconformal mappings: $\{\eta_i^\Psi\}_{i=1}^\ell$ and $\{\eta_i\}_{i=1}^\ell$. The definition of the mappings $\eta_i^\Psi$, $\eta_i$ depend on the objects $\varepsilon$, $K$, $(D_i)_{i=1}^k$, $f$ as fixed in Notation \ref{notation_Section4}, but not on the parameter $n$ in (\ref{newbdefn}). Thus we record the trivial but important observation that the mappings $\{\eta_i^\Psi\}_{i=1}^\ell$ and $\{\eta_i\}_{i=1}^\ell$ are quasiconformal with a constant independent of $n$.
\end{rem}

\section{Annular Interpolation Between a Proper Mapping and a Power Map}\label{extension_section2}

Recall that we have fixed $\varepsilon>0$, a compact set $K$, disjoint analytic domains $(D_i)_{i=1}^k$ so that $U:=\cup_iD_i$ contains $K$, and $f$ holomorphic in a neighborhood of $\overline{U}$ with $||f||_{\overline{U}}<1$. The curves $\{\Gamma_i\}_{i=1}^{k-1}$ connect the domains $(D_i)_{i=1}^k$, and $\Omega$ is a component of the complement of $\overline{U}\cup\cup_{i=1}^{k-1}\Gamma_i$ with $\tau: \Omega\rightarrow\mathbb{D}^*$ conformal. Recall that the domain $\Omega_n'$ was defined in Theorem \ref{complicated_folding_adjustment} and Notation \ref{modified_domain} by removing from $\Omega$ a collection of trees rooted at the vertices along $\partial\Omega$, and the map $\psi\circ\lambda\circ\tau$ maps $\Omega_n'$ onto $\mathbb{D}^*$ (see Proposition \ref{first_folding_adjustment} and Theorem \ref{complicated_folding_adjustment}). 

\begin{notation} Recall from Notation \ref{ell_defn} that $\partial D_i\cap\partial\Omega\not=\emptyset$ if and only if $1\leq i \leq \ell$. Hence exactly $\ell-1$ of the curves $(\Gamma_i)_{i=1}^{k-1}$ intersect $\partial\Omega$. By relabelling the $(\Gamma_i)_{i=1}^{k-1}$ if necessary, we may assume $\Gamma_j$ intersects $\partial\Omega$ if and only if $1\leq j \leq\ell-1$.
\end{notation}

Let $m=m(n)$ be as in Theorem \ref{complicated_folding_adjustment}. To prove our main results, we will need to modify $z\mapsto z^m$ in $\mathbb{D}^*$ so that, roughly speaking, $(z\mapsto z^m)\circ\psi\circ\lambda\circ\tau(z)$ agrees with the proper mappings $B_n$ (see Definition \ref{many_comps_vertex_defn}) along $\partial D_i$. This is done in Theorem \ref{zmadjustment} below (see \cite{2021arXiv210104219B} for a related result). Its proof uses the following. 

\begin{prop}\label{lengthmultprop} Suppose $\phi_1$, $\phi_2$ are $C^1$ homeomorphisms of a $C^1$ Jordan arc $e$ such that: 
\begin{enumerate} \item $\phi_1(e)=\phi_2(e)$, \item $\phi_1$, $\phi_2$ agree on the two endpoints of $e$, and \item $|\phi_1'(z)|=|\phi_2'(z)|$ for all $z\in e$. \end{enumerate} Then $\phi_1=\phi_2$ on $e$. 
\end{prop}

\noindent The proof of Proposition \ref{lengthmultprop} is a consequence of the Fundamental Theorem of Calculus and is left to the reader. 

\begin{thm}\label{zmadjustment} For every $n$, there exists a locally univalent $K$-quasiregular mapping $h_n: \mathbb{D}^* \rightarrow\mathbb{D}^*$ so that:
\begin{enumerate}
\item $h_n(z)=z^m$ for $|z|\geq\sqrt[\leftroot{-2}\uproot{2}m]{2}$ where $m:=m(n)$ is as in Theorem \ref{complicated_folding_adjustment},
\item $h_n\circ\psi\circ\lambda\circ\tau(z)=B_n(z)$ for every $z\in\partial D_i$ and $1\leq i \leq l$, and
\item $K$ is independent of $n$.
 \end{enumerate}
\end{thm}

\begin{proof} Fix the standard branch of $\log$. Given an edge $e\in\partial D_i$, we have by Theorem \ref{complicated_folding_adjustment} that
\begin{equation}\label{linesegment} \log\circ\psi\circ\lambda\circ\tau(e) = \{0\}\times\left[\frac{j\pi}{m},\frac{(j+1)\pi}{m}\right] \textrm{ for some } 0\leq j \leq 2m-1. \end{equation} 
Denote the vertical line segment in (\ref{linesegment}) by $v_e$. Let $f: v_e \mapsto e$ be a length-multiplying, $C^1$ homeomorphism so that $f^{-1}$ agrees with $\log\circ\psi\circ\lambda\circ\tau$ on the two endpoints of $e$. Consider the maps:
\begin{equation}\label{firstmap} z\mapsto mz \textrm{ for } z\in \left\{\frac{\log2}{m}\right\}\times\left[\frac{j\pi}{m},\frac{(j+1)\pi}{m}\right], \end{equation} 
\begin{equation}\label{secondmap} z\mapsto \log\circ B_n\circ f  \textrm{ for } z\in \{0\}\times\left[\frac{j\pi}{m},\frac{(j+1)\pi}{m}\right]. \end{equation}
For each $1\leq i \leq l$, the proper mappings $B_n$ are orientation-preserving on the unique outer boundary component of $D_i$, and orientation-reserving on all other boundary components of $D_i$. This implies that  we may choose the branch  of $\log$ in (\ref{secondmap}) so that the images of (\ref{firstmap}) and (\ref{secondmap}) are horizontal translates of one another (recall $B_n(e)$ is a circular arc of angle $\pi$), and the derivative of (\ref{secondmap}) is strictly positive for all $z\in v_e$. Since the derivative of (\ref{firstmap}) is also strictly positive, this means the linear interpolation between (\ref{firstmap}) and (\ref{secondmap}) is a homeomorphism.

By (\ref{first_car_conc}), we have that $|B_n'|$ is comparable at all points of $e$ with constant independent of $e$ and $n$. Thus, since $f$ is length-multiplying and $\log$ is length-multiplying on Euclidean circles centered at $0$, we conclude that the derivative of (\ref{secondmap}) is comparable to $m$ at all points of $v_e$ with constant independent of $e$ and $n$. Thus, we conclude that the linear interpolation between (\ref{firstmap}) and (\ref{secondmap}) in the rectangle 
\begin{equation}\label{rectangle} \left[0,\frac{\log2}{m}\right]\times\left[\frac{j\pi}{m},\frac{(j+1)\pi}{m}\right] \end{equation} 
is $K$-quasiconformal with $K$ independent of $e$ and $n$ (see, for instance, Theorem A.1 of \cite{MR4008367}). Denote the linear interpolation by $\phi$ (see Figure \ref{fig:Defn_Phi.eps}). 

\begin{figure} 
\includegraphics[scale = .2]{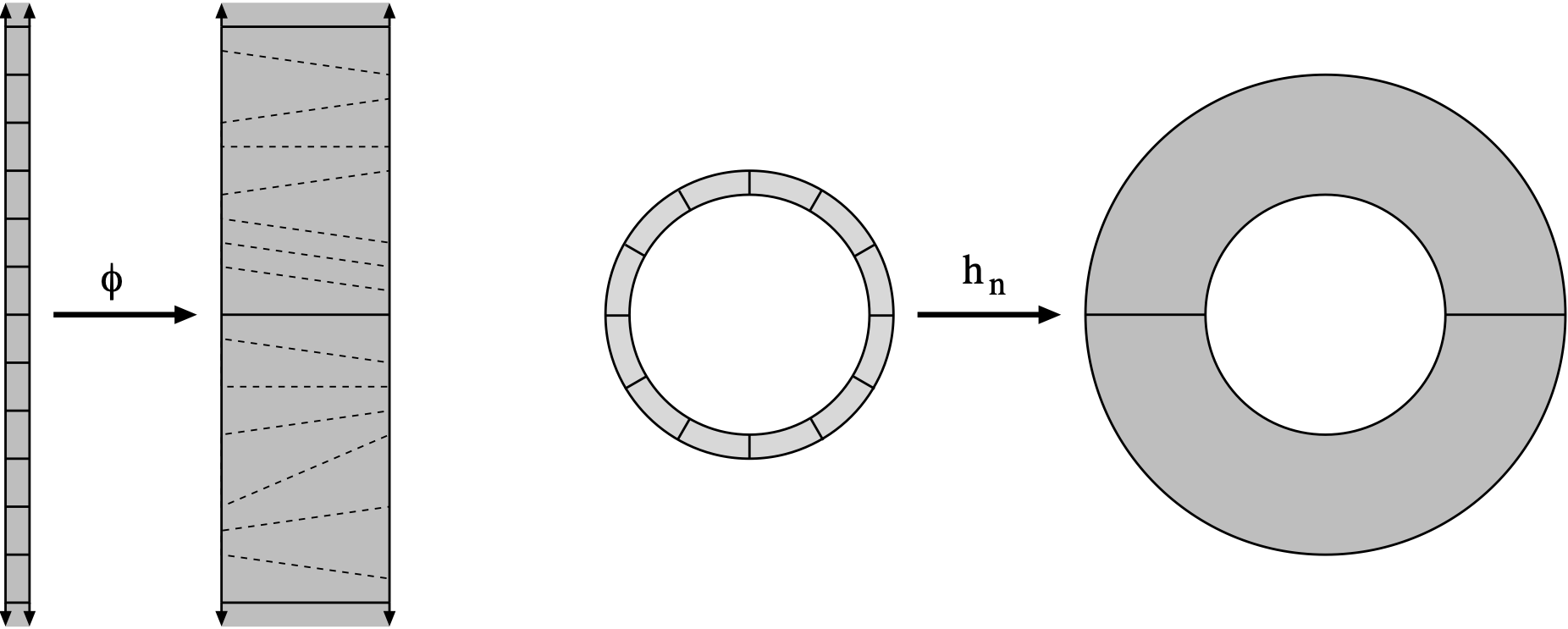}
\caption{Illustrated is the proof of Theorem \ref{zmadjustment}. In logarithmic coordinates the desired interpolation is denoted $\phi$, and $h_n$ is then defined by (\ref{hnfirstdefn}) and (\ref{hnseconddefn}).}
 \label{fig:Defn_Phi.eps}
\end{figure}

 We define
\begin{equation}\label{hnfirstdefn} h_n:=\exp\circ\phi\circ\log \textrm{ in } \{z \in \mathbb{D}^* : z/|z| \in \psi\circ\lambda\circ\tau(e) \}\cap\{|z|\leq\sqrt[\leftroot{-2}\uproot{2}m]{2}\}. \end{equation}
The equation (\ref{hnfirstdefn}) defines $h_n(z)$ for $z$ in $\{ z : 1\leq|z|\leq\sqrt[\leftroot{-2}\uproot{2}m]{2}\}$ and sharing a common angle with the image under $\psi\circ\lambda\circ\tau$ of an edge on some $\partial D_i$. We finish the definition of $h_n$ by simply setting:
\begin{equation}\label{hnseconddefn} h_n(z):=z^m \textrm{ in } \{z \in \mathbb{D}^* : z/|z| \in \psi\circ\lambda\circ\tau(\partial\Omega_n'\setminus\left(\cup_i\partial D_i\right)) \}. \end{equation}
The conclusion (1) now follows by definition of $h_n$, and (3) follows since $h_n$ is a composition of holomorphic mappings and a $K$-quasiconformal interpolation where we have already noted that $K$ is independent of $n$. 

We now show that conclusion (2) follows from Proposition \ref{lengthmultprop}. Fix an edge $e$ on $\partial D_i$. Recall $v_e:=\log\circ\psi\circ\lambda\circ\tau(e)$. Thus, by (\ref{secondmap}) and (\ref{hnfirstdefn}) we have that:
\begin{equation}\label{hncomplicated} h_n\circ\psi\circ\lambda\circ\tau=B_n\circ f\circ\log\circ\psi\circ\lambda\circ\tau \textrm{ on } e. \end{equation}
First note that (\ref{hncomplicated}) agrees set-wise with $B_n$ on $e$ and at the endpoints of $e$. The map $\psi\circ\lambda\circ\tau$ is length-multiplying (by Proposition \ref{first_folding_adjustment}(3) and Theorem \ref{complicated_folding_adjustment}(2)), $\log$ is length-multiplying on the circular segment $\psi\circ\lambda\circ\tau(e)$, and $f$ is length-multiplying by definition. Thus the modulus of the derivative of $f\circ\log\circ\psi\circ\lambda\circ\tau$ is constant on $e$, and so the derivatives of (\ref{hncomplicated}) and $B_n$ have the same modulus at each point of $e$. Conclusion (2) now follows from Proposition \ref{lengthmultprop}. 
\end{proof}




\section{Joining Different Types of Boundary Arcs: the Map $E_n$}\label{extension_section3}

Recall that in Section \ref{extension_section} we defined the maps $\eta_i^\Psi$, $\eta_i$ where $1\leq i \leq \ell$, and in Section \ref{extension_section2} we defined the map $h_n$ for all $n\in\mathbb{N}$. In this section we define a map $E_n$ in $\mathbb{D}^*$ which is roughly given by either $z\mapsto \eta_i^\Psi\circ h_n(z)$ or $z\mapsto \eta_i(z)\circ h_n(z)$, where $i$ is allowed to depend on $\textrm{arg}(z)$ and which of $\eta_i^\Psi$, $ \eta_i$ we post-compose $h_n$ with is also allowed to depend on $\textrm{arg}(z)$. Thus, we will need a way to interpolate between the definitions of $\eta_i^\Psi$, $\eta_i$, for different $i$. The interpolation regions are defined in Definition \ref{marked_edges} below, and the map $E_n$ in Proposition \ref{quadrilateral_defn_extension_prop}. It will be useful to keep Figure \ref{fig:ThreeEdges} in mind for the remainder of this section.

\begin{definition}\label{marked_edges} Mark one edge $e_i$ on $\Gamma_i$ for each $1\leq i \leq\ell-1$. Label the $\ell$ components of $\partial\Omega_n'\setminus\cup_ie_i$ as $(G_i)_{i=1}^\ell$, where $\partial D_i\subset G_i$. Let 
\begin{enumerate} 
\item $\mathcal{J}_i^{D}$ denote those edges in $\partial D_i$,
\item $\mathcal{J}_i^{G}$ denote those edges in $G_i\setminus\mathcal{J}_i^{D}$,
\item $\mathcal{J}^{e}$ denote the edges $(e_i)_{i=1}^{\ell-1}$. 
\end{enumerate} 
In other words, $\mathcal{J}_i^{D}$ are the edges shared by $\partial\Omega_n'$ and $\partial D_i$, $\mathcal{J}^{e}$ consists of $\ell-1$ edges: one on each of the curves $(\Gamma_i)_{i=1}^\ell$, and ${\mathcal J}_i^G$ are the remaining edges on $G_i$. Thus we have: \begin{equation}\nonumber \partial\Omega_n'=\mathcal{J}^{e}  \cup \bigcup_i \left( \mathcal{J}_i^{D} \cup \mathcal{J}_i^{G} \right). \end{equation}  For $z \in \mathbb{D}^*$, we define: 
\begin{equation}\label{definition_of_E} E_n(z):= \begin{cases} 
	 \eta_i^\psi\circ h_n(z) & \textrm{ if } \hspace{2mm} z/|z|\in \psi\circ\lambda\circ\tau(\mathcal{J}_i^D) \\
	\eta_i\circ h_n(z) & \textrm{ if } \hspace{2mm} z/|z|\in \psi\circ\lambda\circ\tau(\mathcal{J}_i^G) \\	 
   \end{cases} \end{equation}
\end{definition}

\begin{figure} 
\includegraphics[scale = .25]{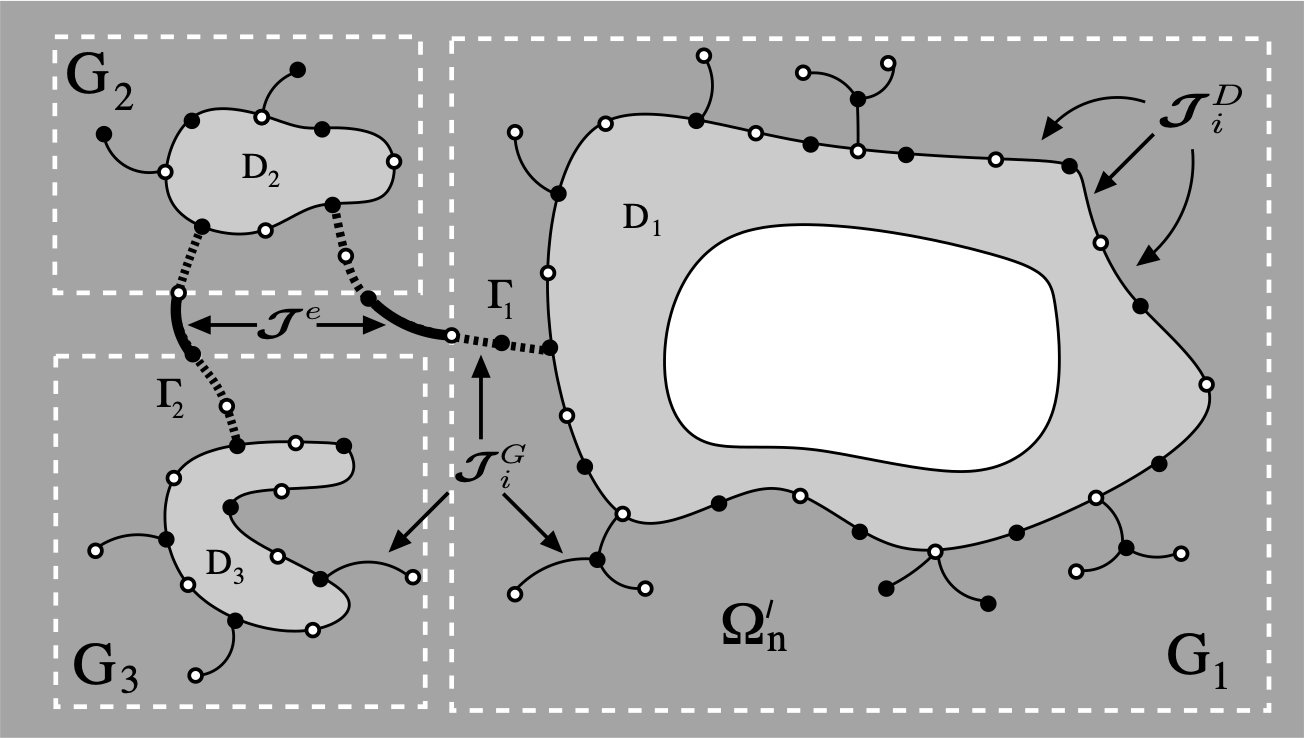}
\caption{Illustrated is Definition \ref{marked_edges}. The curves $\Gamma_1$, $\Gamma_2$ are depicted as black dotted lines, except for the edges $e_1\subset\Gamma_1$, $e_2\subset\Gamma_2$ which are in thick black. }
 \label{fig:ThreeEdges}
\end{figure} 

\noindent It remains to define $E_n(z)$ for $z\in\mathbb{D}^*$ satisfying $z/|z|\in\psi\circ\lambda\circ\tau(\mathcal{J}^e)$. We do so in the following Proposition.

\begin{prop}\label{quadrilateral_defn_extension_prop} The map $E_n$ extends to a locally univalent $K$-quasiregular mapping $E_n:\mathbb{D}^*\rightarrow\mathbb{C}$ satisfying $E_n(z)=z^{m}$ for $|z|\geq \sqrt[\leftroot{-2}\uproot{2}m]{2}$, where $m=m(n)$ is as in Theorem \ref{complicated_folding_adjustment}. Moreover, $K$ does not depend on $n$.
\end{prop}

\begin{proof} Consider (\ref{definition_of_E}). Note that if $E_n$ is defined at $z$ and $|z|\geq \sqrt[\leftroot{-2}\uproot{2}m]{2}$, then $E_n(z)=z^m$ by Theorem \ref{zmadjustment}(1) and (\ref{first_desired}), (\ref{first_eta_i_R}). Thus, setting $E_n(z):=z^{m}$ for $|z|\geq \sqrt[\leftroot{-2}\uproot{2}m]{2}$ extends the definition of $E_n$.

 It remains to extend the definition of $E_n$ to: \begin{equation}\label{quadrilaterals} \{ z : 1\leq |z|\leq \sqrt[\leftroot{-2}\uproot{2}m]{2} \textrm{ and } z/|z|\in\psi\circ\lambda\circ\tau(e_i) \}\textrm{, for } 1\leq i \leq\ell-1. \end{equation}
Each of the $\ell-1$ sets in (\ref{quadrilaterals}) consists of $2$ quadrilaterals which we denote by $\mathcal{Q}_i^{\pm}$. The curve $\Gamma_i$ connects two distinct elements of $(D_i)_{i=1}^{\ell-1}$. In order to avoid complicating notation significantly, we will assume without loss of generality that $\Gamma_i$ connects $D_i$ to $D_{i+1}$. Let $\gamma_i\subset 2\mathbb{D}$ be a smooth Jordan arc connecting $\eta_i^\Psi(1)$ to $\eta_{i+1}^\Psi(-1)$ (see Figure \ref{fig:Defn_gi.eps}). Moreover, by (\ref{third_desired}), we can choose $\gamma_i$ so that the union of the arcs 
\begin{equation}\label{quad_defn_} \eta_i^\Psi([1,2])\textrm{, }2\mathbb{T}\cap\mathbb{H}\textrm{, } \eta_{i+1}^\Psi([-2,-1])\textrm{, }  \gamma_i \end{equation}
forms a topological quadrilateral we denote by $Q_i^+$ (in particular none of the arcs in (\ref{quad_defn_}) intersect except at common endpoints).

Define a quasisymmetric homeomorphism $g_i^+: \partial Q_i^+ \rightarrow \partial (A(1,2) \cap \mathbb{H})$ (see Figure \ref{fig:Defn_gi.eps}) by 
\begin{equation}\nonumber g_i^+(z)=z \textrm{ for } z\in 2\mathbb{T} \end{equation}
\begin{equation}\nonumber g_i^+(z)=(\eta_i^\Psi)^{-1}(z) \textrm{ for } z\in\eta_i^\Psi([1,2]) \end{equation}
\begin{equation}\nonumber g_i^+(z)=(\eta_{i+1}^\Psi)^{-1}(z) \textrm{ for } z\in\eta_{i+1}^\Psi([-2,-1]), \end{equation} 
and extending $g_i^+$ to a quasisymmetric homeomorphism of $\gamma_i$ to $\mathbb{T}\cap\mathbb{H}$. The mapping $g_i^+$ extends to a quasiconformal homeomorphism $g_i^+: Q_i^+ \rightarrow A(1,2)\cap \mathbb{H}$ (see Lemma 2.24 of \cite{MR3445628}). We define $E_n(z):=(g_i^+)^{-1}(z^m)$ for $z\in \mathcal{Q}_i^+$. A similar definition of $g_i^-: Q_i^- \rightarrow A(1,2) \cap -\mathbb{H}$ is given (using the same curve $\gamma_i$) so that $$g_i^+(z)=g_i^-(\overline{z}) \textrm{ for } z\in\mathbb{T}\cap\mathbb{H}.$$
We let $E_n(z):=(g_i^-)^{-1}(z^m)$ for $z\in \mathcal{Q}_i^-$. 

\begin{figure} 
\includegraphics[scale = .2]{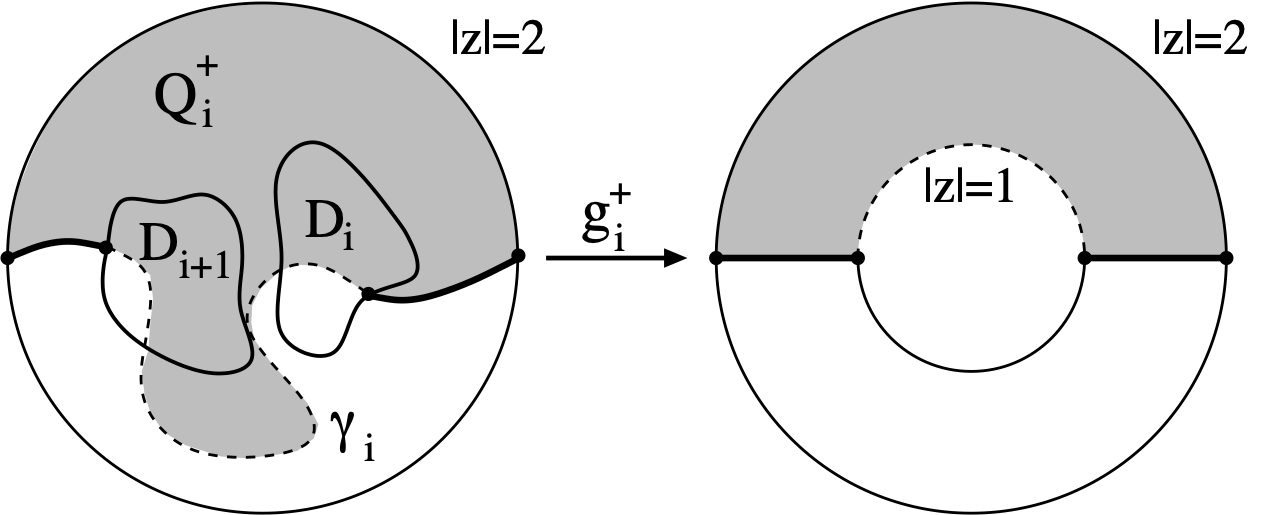}
\caption{Illustrated is the quadrilateral $Q_i^+$ and the map $g_i^+$ in the proof of Proposition \ref{quadrilateral_defn_extension_prop}.}
 \label{fig:Defn_gi.eps}
\end{figure} 

To summarize, we have defined $E_n$ in each of the three regions
\begin{align}\label{finaldefnregion1} \{ z\in\mathbb{D}^* : z/|z| \in \psi\circ\lambda\circ\tau(\mathcal{J}_i^D) \}, \\ 
\label{finaldefnregion2} \{ z\in\mathbb{D}^* : z/|z| \in \psi\circ\lambda\circ\tau(\mathcal{J}_i^G) \},  \\
\label{finaldefnregion3}   \{ z\in\mathbb{D}^* : z/|z| \in \psi\circ\lambda\circ\tau(\mathcal{J}_i^e) \},
\end{align}
Indeed, the definition of $E_n$ in (\ref{finaldefnregion1}) and (\ref{finaldefnregion2}) was given already in (\ref{definition_of_E}), and in this proof we have defined $E_n$ in (\ref{finaldefnregion3}). The definitions of $E_n$ in each of (\ref{finaldefnregion1}), (\ref{finaldefnregion2}), (\ref{finaldefnregion3}) agree along any common boundary, and thus by removability of analytic arcs for quasiregular mappings, it follows that $E_n$ is quasiregular on $\mathbb{D}^*$. Moreover, $E_n$ has no branched points in $\mathbb{D}^*$, and hence $E_n$ is locally quasiconformal. The dilatation of the map $E_n$ depends only on the dilatation of $h_n$ (which is independent of $n$ by Theorem \ref{zmadjustment}(3)) and the dilatations of the the finite collection of quasiconformal maps used in its definition: $\eta_i^\Psi$, $\eta_i$, $g_i^+$, $g_i^-$, and hence we may take $K$ independent of $n$.
\end{proof}

\section{Defining $g_n$ in $\Omega_n'$}\label{extension_section4}


First we recall our setup. We have fixed $\varepsilon>0$, a compact set $K$, disjoint, analytic domains $(D_i)_{i=1}^k$ so that $K\subset U:=\cup_iD_i$, and $f$ holomorphic in a neighborhood of $\overline{U}$ with $||f||_{\overline{U}}<1$. We defined curves $\{\Gamma_i\}_{i=1}^{k-1}$ connecting the domains $(D_i)_{i=1}^k$, and we denoted by $\Omega$ a component of the complement of $\overline{U}\cup\cup_{i=1}^{k-1}\Gamma_i$ with $\tau: \Omega\rightarrow\mathbb{D}^*$ conformal. The domain $\Omega_n'$ is contained in $\Omega$, and $\psi\circ\lambda\circ\tau$ maps $\Omega_n'$ onto $\mathbb{D}^*$. In Section \ref{extension_section3} we defined the map $E_n$.

\begin{definition}\label{g_outside_defn} We define the mapping $g_n: \Omega_n' \rightarrow \Chat$ by 
\begin{equation}\label{g_n_comp_defn}  g_n:=E_{n}\circ\psi\circ\lambda\circ\tau.\end{equation}
\end{definition}


\noindent We will now record at which points the function $g_n|_{\Omega_n'}$ is locally $n:1$ for $n>1$.

\begin{definition} Let $g$ be a quasiregular function, defined in a neighborhood of a point $z\in\mathbb{C}$. We say that $z$ is a \emph{branched point} of $g$ if for any sufficiently small neighborhood $U$ of $z$, the map $g|_U$ is $n:1$ onto its image for $n>1$. We say $w\in\mathbb{C}$ is a \emph{branched value} of $g$ if $w=g(z)$ for a branched point $z$ of $g$. We denote the branched points of a quasiregular mapping $g$ by $\textrm{BP}(g)$, and the branched values by $\textrm{BV}(g)$.
\end{definition}

\begin{rem} Recall that in Notation \ref{omega_def_setup} we fixed a point $p\in\Omega$ satisfying $\tau(p)=\infty$.
\end{rem}

\begin{prop}\label{branched_points_location} The mapping $g_n: \Omega_n' \rightarrow \mathbb{C}$ of Definition \ref{g_outside_defn} is $K$-quasiregular and $C$-vertex supported for $K$, $C$ independent of $n$. Moreover, $g_n^{-1}(\infty)=\{p\}$, \begin{equation}\label{b.p.location} \emph{BP}(g_n) \subset \bigcup_{e\in\partial\Omega_n} \{ z : \dist(z,e) < C\cdot\emph{diam}(e) \}\emph{, and } \end{equation}
\begin{equation} \label{b.v.location.equation} \emph{BV}(g_n) \subset \bigcup_{i=1}^{k} \Psi_i(\mathbb{T}). \end{equation}
\end{prop}

\begin{proof} Since each of the mappings in the composition (\ref{g_n_comp_defn}) are $K$-quasiregular and $C$-vertex supported for $K$, $C$ independent of $n$, the same is true of $g_n$.  The only points where the mapping $g_n$ is locally $l:1$ for $l>1$ are a subset of the vertices of the graph $\partial\Omega_n'$. By Theorem \ref{complicated_folding_adjustment}, the vertices of $\partial\Omega_n'$ all lie in \begin{equation}\nonumber \bigcup_{e\in\partial\Omega_n} \{ z : \dist(z,e) < C\cdot\textrm{diam}(e) \}.\end{equation} Thus, (\ref{b.p.location}) is proven. Moreover, any vertex of $\partial\Omega_n'$ is mapped to a point on one of the curves $\Psi_i(\mathbb{T})$ by $g_n$. Hence, (\ref{b.v.location.equation}) follows since $\textrm{BV}(g_n) = g_n( \textrm{BP}(g_n) )$.
It remains to show: \begin{equation}\label{one_pole} g_n^{-1}(\infty)=\{p\}. \end{equation} Indeed, note that $E_{n}\circ\psi\circ\lambda$ fixes $\infty$ and has no finite poles. The map $\tau: \Omega\rightarrow\mathbb{D}^*$ is conformal and hence only one point $p$ is mapped to $\infty$. The relation (\ref{one_pole}) now follows.

\end{proof}

\noindent It will be useful to record the following result.

\begin{prop} Let $r>1$. Then for all sufficiently large $n$, we have:
\begin{equation}\label{zmoutsidenbhd} g_n(z)=\tau(z)^m \textrm{ for any } z\in\tau^{-1}(\{z: |z|>r\}). \end{equation}
\end{prop}

\begin{proof} Consider the functional equation (\ref{g_n_comp_defn}) defining $g_n$. The maps $\lambda$, $\psi$ are vertex-supported, and moreover $\lambda$ (respectively, $\psi$) is the identity outside of the support of $\lambda_{\overline{z}}$, (respectively, $\psi_{\overline{z}}$). By Proposition \ref{max_diam_omega}, we therefore have that $\psi\circ\lambda(z)=z$ if $z\in\tau^{-1}(\{z: |z|>r\})$ and $n$ is sufficiently large. The relation (\ref{zmoutsidenbhd}) now follows from (\ref{g_n_comp_defn}) and Theorem \ref{zmadjustment}(1) since $m\rightarrow\infty$ as $n\rightarrow\infty$. 
\end{proof}


\begin{rem}\label{rem_long_notation} As in Remark \ref{dependence_notation}, we note that our Definition \ref{g_outside_defn} of $g_n$ is determined by a choice of the objects $K$, $U$, $\mathcal{D}$, $f$, $\varepsilon$, $\Omega$, $p$ we fixed in Notations \ref{notation_Section4} and \ref{omega_def_setup}. When we wish to emphasize this dependence, we will write $g_n(K, U, \mathcal{D}, f, \varepsilon, \Omega, p)$. In particular, it will be useful in the next section to think of $g_n$ as a function taking as input any choice of  $K$, $U$, $\mathcal{D}$, $f$, $\varepsilon$, $\Omega$, $p$ satisfying the conditions in Notations \ref{notation_Section4}, \ref{omega_def_setup}, and outputting (via Definition \ref{g_outside_defn}) a quasiregular function $g_n(K, U, \mathcal{D}, f, \varepsilon, \Omega, p)$ defined on $\Omega_n'$. 
\end{rem}

\section{Verifying $g_n$ is Quasiregular on $\Chat$}\label{quasiregular_approx}

In this section we combine our efforts in Sections \ref{blaschke_approximation}-\ref{extension_section4} to define an approximant $g_n: \Chat\rightarrow\Chat$ of a given $f$. The approximant $g_n$ will not be holomorphic as required in Theorems \ref{main_theorem} and \ref{rational_runge}, but we will solve this problem in the next section by applying the Measurable Riemann Mapping Theorem. We fix the following for Sections \ref{quasiregular_approx}-\ref{main_proofs_section}.

\begin{notation}\label{last_section_notation} Fix $K$, $f$, $\mathcal{D}$, $\varepsilon$, $P$ as in the statement of Theorem \ref{rational_runge}. Denote by $U$ the neighborhood of $K$ in which $f$ is holomorphic. Define 
\begin{equation} P':=\{ p \in P : p\textrm{ is contained in a component } V \textrm{ of } \Chat\setminus K \textrm{ such that } V\not\subseteq U\}. \end{equation}
Compactness of $K$ implies that $U$ contains all but finitely many components of $\Chat\setminus K$, and so the set $P'$ is finite. Moreover, $P'$ does not depend on $\varepsilon$. By shrinking $U$ if necessary, we may assume that:
\begin{enumerate}
\item $U\cap P'=\emptyset$,
\item $P'$ contains exactly one point in each component of $\Chat\setminus U$,
\item $f$ is holomorphic in a neighborhood of $U\subset\mathcal{D}$, and 
\item the components of $U$ are a finite collection of analytic Jordan domains $(D_i)_{i=1}^k$ so that (\ref{closetoboundaryassumption}) holds for each $D_i$. 
\end{enumerate}
Let $K'$ be a compact set such that $K\subset \textrm{int}(K')\subset K' \subset U$. We will assume for now that $||f||_{\overline{U}}<1$.
\end{notation}

We now define a quasiregular approximation $g_n$ of $f$ by applying the construction of Section \ref{blaschke_approximation} in each $D_i$, and by applying the folding construction of Sections \ref{folding_section}-\ref{extension_section4} in each complementary component of $\cup_i(\overline{D_i}\cup\Gamma_i)$:





\begin{definition}\label{g_final_defn} For every $n$, we define a quasiregular mapping $g_n$ as follows. Recalling Remark \ref{dependence_notation}, we first set 
\begin{equation}\label{definition_of_g_prelim} g_n:= g_n(\varepsilon, K'\cap D_i, D_i, f|_{D_i})  \textrm{ in } D_i \textrm{ for } 1\leq i \leq k \end{equation}
The equation (\ref{definition_of_g_prelim}) defines the curves $(\Gamma_i)_{i=1}^k$ by way of Proposition \ref{existence_of_curves}, and we enumerate the components of 
\begin{equation}\nonumber \Chat\setminus\left( \overline{U} \cup \bigcup_{i=1}^{k-1}\Gamma_i \right) \end{equation}
by $(\Omega(i))_{i=1}^\ell$. Recalling Remark \ref{rem_long_notation} and Notation \ref{modified_domain}, we extend the definition of $g_n$ to the open set
 \begin{equation}\label{boundaries} \Omega:=\Chat\setminus \left( \bigcup_{i=1}^\ell \partial\Omega_n'(i) \right) \end{equation}
 by the formula
 \begin{equation}\label{definition_of_g_notprelim} g_n:= g_n(K', U, \mathcal{D}, f, \varepsilon, \Omega(i), P'\cap\Omega(i)) \textrm{ in } \Omega_n'(i) \textrm{ for } 1\leq i \leq\ell. \end{equation}
\end{definition}

\begin{prop}\label{C_and_K} The quasiregular function $g_n$ is $C$-vertex supported and $K$-quasiregular for $C$, $K$ independent of $n$.
\end{prop}
\begin{proof} For $g_n|_{\Omega_n'(i)}$ this is exactly Proposition \ref{branched_points_location}, and so the conclusion follows since $g_n$ is holomorphic in $U$.
\end{proof}


The function $g_n$ is now defined on all of $\Chat$ except for the edges of each $\partial\Omega_n'(i)$. We show in Propositions \ref{edgesonD}, \ref{edgesonO} below that $g_n$ in fact extends continuously across each edge of $\partial\Omega_n'(i)$, and deduce in Corollary \ref{acrossedgescor} that $g_n$ extends quasiregularly across $\partial\Omega$. 

\begin{prop}\label{edgesonD} The $K$-quasiregular function $g_n: \Omega \rightarrow \Chat$ extends to a continuous function $g_n: \Omega\cup e \rightarrow \Chat$ for any edge $e\subset\partial\Omega\cap\partial U$.
\end{prop}

\begin{proof} Let $i$ be so that $e\subset\partial D_i$ and denote the unique element of $(\Omega(i))_{i=1}^\ell$ that contains $e$ on its boundary by $\Omega(j)$. Recall by Definitions \ref{quasiregular_disc_definition} and \ref{g_outside_defn} that 
\begin{equation}\label{recallgnDdefn} g_n|_{D_i}=\Psi_i\circ B_n, \end{equation} 
\begin{equation}\label{recalloutsidedefn} {g_n}|_{\Omega_n'(j)} = E_n\circ\psi\circ\lambda\circ\tau. \end{equation}
By Theorem \ref{zmadjustment}(2) we have
\begin{equation}\nonumber h_n\circ\psi\circ\lambda\circ\tau=B_n \textrm{ on } e. \end{equation}
By (\ref{second_desired}) and the definition (\ref{definition_of_E}) of $E_n$, it follows from (\ref{recalloutsidedefn}) that 
\begin{equation}\nonumber {g_n}|_{\Omega_n'(j)}(z)=\Psi_i\circ B_n(z) \textrm{ for } z\in e, \end{equation}
in other words ${g_n}|_{\Omega_n'(j)}$ and $g_n|_{D_i}$ agree pointwise on $e$.
\end{proof}

\begin{prop}\label{edgesonO} The $K$-quasiregular function $g_n: \Omega \rightarrow \Chat$ extends to a continuous function $g_n: \Omega\cup e \rightarrow \Chat$ for any edge $e\subset\partial\Omega\cap(\cup_{i=1}^\ell\Omega(i))$.
\end{prop}

\begin{proof} Let $j$ be so that $e\subset \partial\Omega_n'(j)$, and as in the proof of Proposition \ref{edgesonD}, recall that 
\begin{equation}\label{recalloutsidedefn2} {g_n}|_{\Omega_n'(j)} = E_n\circ\psi\circ\lambda\circ\tau. \end{equation}
Let $x\in e$. There are two limits 
\begin{equation}\nonumber \lim_{\Omega_n'(j)\ni z\rightarrow x}\psi\circ\lambda\circ\tau(z), \end{equation} each lying on the unit circle.
Denote them by $\zeta_\pm$. By Theorem \ref{complicated_folding_adjustment}(3), \begin{equation}\nonumber \zeta_+^m=\overline{\zeta_-^m}. \end{equation} Thus, by (\ref{conjugacy_relation}) and (\ref{definition_of_E}), we conclude that there is a unique limit 
\begin{equation}\nonumber \lim_{\Omega_n'(j)\ni z\rightarrow x}E_n\circ\psi\circ\lambda\circ\tau(z). \end{equation} 
Hence, setting 
\begin{equation}\nonumber g_n(x):=\lim_{\Omega_n'(j)\ni z\rightarrow x}E_n\circ\psi\circ\lambda\circ\tau(z) \end{equation} 
defines a continuous extension of $g_n$ across the edge $e$.
\end{proof}

\begin{cor}\label{acrossedgescor} The $K$-quasiregular function $g_n: \Omega \rightarrow \Chat$ extends to a $K$-quasiregular function $g_n: \Chat\rightarrow\Chat$.
\end{cor}

\begin{proof} The set $\Chat\setminus\Omega=\partial\Omega$ consists of a finite collection of analytic arcs: the edges of the graphs $\partial\Omega_n'(i)$ over $1\leq i \leq\ell$. Thus, by removability of analytic arcs for quasiregular mappings, it suffices to show that $g_n: \Omega \rightarrow \Chat$ extends continuously across each such edge. There are two types of edges to check: those that lie on the boundary of a domain $D_i$, and those that lie in the interior of a domain $\Omega(i)$. We have already checked continuity across both types of edges in Propositions \ref{edgesonD}, \ref{edgesonO}, and so the proof is complete.
\end{proof}

\section{Proof of the Main Theorems}\label{main_proofs_section}

In Section \ref{main_proofs_section} we prove Theorems \ref{main_theorem} and \ref{rational_runge}. Recall that in Section \ref{quasiregular_approx} we fixed the objects  $K$, $f$, $\mathcal{D}$, $\varepsilon$, $P$ as in Theorem \ref{rational_runge} (see Notation \ref{last_section_notation}), and we defined a quasiregular approximation $g_n$ to $f$ in Definition \ref{g_final_defn}. We also showed in Section \ref{quasiregular_approx} that $g_n$ in fact extends to a quasiregular function $g_n: \Chat\rightarrow\Chat$. Now we apply the MRMT below in Definition \ref{polynomial_defn} to obtain the rational maps $r_n: \Chat \rightarrow \Chat$ which we will prove satisfy the conclusions of Theorems \ref{main_theorem} and \ref{rational_runge} for large $n$.


\begin{definition}\label{polynomial_defn} The mapping $g_n$ induces a Beltrami coefficient $\mu_n:=(g_n)_{\overline{z}}/(g_n)_z$, which, by way of the MRMT, defines a quasiconformal mapping $\phi_n:\Chat\rightarrow\Chat$ such that $r_n:=g_n\circ\phi_n^{-1}$ is holomorphic. We normalize $\phi_n$ so that $\phi_n(\infty)=\infty$ and $\phi_n(z)=z+O(1/|z|)$ as $z\rightarrow\infty$. 
\end{definition}

We now begin deducing that for large $n$, the maps $r_n$ satisfy the various conclusions in Theorems \ref{main_theorem} and \ref{rational_runge}.

\begin{prop}\label{rationalpluspoles} The function $r_n$ of Definition \ref{polynomial_defn} is rational, and $r_n^{-1}(\infty)=\phi_n(P')$. In particular, if $K$ is full and $P=\{\infty\}$, then $r_n$ is a polynomial. 
\end{prop}

\begin{proof} The function $r_n$ is holomorphic on $\Chat$ and takes values in $\Chat$: the only such functions are rational. Note that $g_n^{-1}(\infty)\cap\overline{U}=\emptyset$ since $g_n$ is bounded on $\overline{U}$. Thus, by Proposition \ref{branched_points_location} and (\ref{definition_of_g_notprelim}), we have that $g_n^{-1}(\infty)=P'$. Since $r_n:=g_n\circ\phi_n^{-1}$, we conclude that $r_n^{-1}(\infty)=\phi_n(P')$. The last statement of the proposition follows since we normalized $\phi_n(\infty)=\infty$, and the only rational functions with a unique pole at $\infty$ are polynomials.
\end{proof}

\begin{prop}\label{phi_close_to_id_prop} For all $R<\infty$, the mapping $\phi_n$ satisfies: \begin{equation}\label{phi_inequality0} ||\phi_n(z)-z||_{R\cdot\mathbb{D}}\xrightarrow{n\rightarrow\infty}0. \end{equation}
\end{prop}

\begin{proof}
Since $g_n$ is $C$-vertex supported by Proposition \ref{C_and_K}, we conclude from Proposition \ref{max_diam_omega} that  \begin{equation}\label{small_area} \textrm{Area}(\textrm{supp}(\mu_n)) \xrightarrow{n\rightarrow\infty}0. \end{equation} 
The relation (\ref{phi_inequality0}) now follows from (\ref{small_area}) since $||\mu_n||_{L^\infty}\leq K$ for all $n$ by Proposition \ref{C_and_K}.
\end{proof}

\begin{thm}\label{p_n_crit_pts} For all sufficiently large $n$, the mapping $r_n$ satisfies $\emph{CP}(r_n)\subset\mathcal{D}$.
\end{thm}

\begin{proof} By Proposition \ref{max_diam_omega}, we have \begin{equation}\nonumber \max\left\{\textrm{diam}(e) : e \textrm{ is an edge of } \bigcup_{i=1}^\ell\partial\Omega(i) \right\} \xrightarrow{n\rightarrow\infty}0. \end{equation} Thus, since $\mathcal{D}$ is a domain containing $\cup_{i=1}^\ell\partial\Omega(i)$, we have by (\ref{b.p.location}) that $\textrm{BP}(g_n) \setminus U \subset \mathcal{D}$ for large $n$. Since $U \subset \mathcal{D}$ we conclude that $\textrm{BP}(g_n) \subset \mathcal{D}$ for large $n$. The result now follows from Proposition \ref{phi_close_to_id_prop} since $\phi(\textrm{BP}(g_n))=\textrm{CP}(r_n)$.
\end{proof}

\begin{thm}\label{Linftyconclusion} For all sufficiently large $n$, we have \begin{equation}\nonumber ||r_n-f||_{K}<2\varepsilon. \end{equation} \end{thm}

\begin{proof} First we note that since $f$ is uniformly continuous on $K'$, there exists $\delta>0$ so that if $z$, $w\in K'$ and $|z-w|<\delta$, then $|f(z)-f(w)|<\varepsilon$. By Proposition \ref{phi_close_to_id_prop}, we can conclude that \begin{equation}\label{phi_inequality} ||\phi_n(z)-z||_{K}<\textrm{min}(\delta, \dist(K, \partial K')) \end{equation} for all sufficiently large $n$. 

Let $z\in K$ and $w_n:=\phi_n^{-1}(z)$. Then 
\begin{equation} |r_n(z)-f(z)| = |g_n(w_n)-f(z)| \leq |g_n(w_n)-f(w_n)| + |f(w_n)-f(z)|. \end{equation}
It follows from (\ref{phi_inequality}) that $|f(w_n)-f(z)|<\varepsilon$ for sufficiently large $n$. Since $w_n\in K'$ for large $n$, we also have by (\ref{g_n_disc_defn}) that $g_n(w_n):=\Psi_j\circ B_n(w_n)$, whence it follows from Proposition \ref{first_def_B_n} that $|g_n(w_n)-f(w_n)|<\varepsilon$. 
\end{proof}



\begin{thm}\label{third_condition_proof} For all sufficiently large $n$, we have \begin{equation}\nonumber \emph{CV}(r_n) \subset N_\varepsilon\widehat{f(K)} \end{equation}
\end{thm}

\begin{proof} Since \begin{equation}\nonumber \textrm{CV}(r_n)=\textrm{BV}(g_n),  \end{equation} it suffices to show that for every $z\in\textrm{BP}(g_n)$ and sufficiently large $n$, we have \begin{equation}\label{wtsgn} g_n(z)\in N_\varepsilon\widehat{f(K)}. \end{equation}
For $z\in\textrm{BP}(g_n)\setminus U$, it follows from (\ref{b.v.location.equation}) that $g_n(z) \in N_\varepsilon f(K) \subset N_\varepsilon\widehat{f(K)}$. For $z\in\textrm{BP}(g_n)\cap D_i$,  (\ref{wtsgn}) follows from Definition \ref{quasiregular_disc_definition} of $g_n|_{D_i}$.
\end{proof}

\noindent \emph{Proof of Theorem \ref{rational_runge}:} In the special case that $||f||_{K}<1$, we have already proven that the mappings $r_n$ satisfy the conclusions of Theorem \ref{rational_runge} for all sufficiently large $n$. Indeed, Theorem \ref{Linftyconclusion} says that $||r_n-f||_K<2\varepsilon$, conclusion (2) in Theorem \ref{rational_runge} is Theorem \ref{p_n_crit_pts}, and conclusion (3) is Theorem \ref{third_condition_proof}. Conclusion (1) follows from Propositions \ref{rationalpluspoles}, \ref{phi_close_to_id_prop}. The general case follows by applying the above special case to an appropriately rescaled $f$. \hfill $\qed$

\vspace{2.5mm}

\noindent \emph{Proof of Theorem \ref{main_theorem}:} When $K$ is full, we may take $P=\{\infty\}$ and apply Theorem \ref{rational_runge}, in which case Proposition \ref{rationalpluspoles} guarantees that the maps $r_n$ are polynomials. \hfill $\qed$



\vspace{2.5mm}
\begin{thm}\label{Mergelyan} {\bf (Mergelyan$+$)}
Let $K\subset\mathbb{C}$ be full, suppose $f\in C(K)$ is holomorphic in $\emph{int}(K)$, and let $\mathcal{D}$ be a domain containing $K$. For every $\varepsilon>0$, there exists a polynomial $p$ so that $||p-f||_{K}<\varepsilon$ and: 
\begin{enumerate}  
\item $\emph{CP}(p)\subset \mathcal{D}$, 
\item $\emph{CV}(p) \subset N_\varepsilon\widehat{f(K)}$. 
\end{enumerate} \end{thm}

\noindent \emph{Proof}: By the usual version of 
Mergelyan's Theorem, there exists a polynomial $q$ so that $||q-f||_{K}<\varepsilon/2$. Apply Theorem \ref{main_theorem} to $K$, $\mathcal{D}$, $q$, $\varepsilon/2$ to obtain an approximant of $q$ which we denote by $p$. The polynomial $p$ satisfies the conclusions of Theorem \ref{Mergelyan}. \hfill $\qed$

\begin{cor}\label{Weierstrass} {\bf (Weierstrass$+$)}
	Suppose  that $I\subset\reals$ is a closed interval, $f: I \rightarrow\mathbb{R}$ is continuous, and $U$, $V\subset\mathbb{C}$ are planar domains containing $I$, $f(I)$, respectively. Then, for every $\varepsilon>0$,
 there exists a polynomial $p$ with real coefficients so that $\|f-p\|_I \leq \varepsilon$, and 
\begin{enumerate}  
\item $\emph{CP}(p)\subset U$,
\item $ \emph{CV}(p) \subset V$. 
\end{enumerate} 
\end{cor}

\begin{proof} Let $I=[a,b]$, and $f$, $U$, $V$ as in the statement of the corollary. By Theorem \ref{Mergelyan}, there exists a complex polynomial $q$ so that $||q-f||_{[a,b]}<\varepsilon/2$. The real polynomial
\begin{equation}\nonumber Q(z):=\frac{q(z)+\overline{q(\overline{z})}}{2} \end{equation}
satisfies $Q(x)=\textrm{Re}(q(x))$ for $x\in\mathbb{R}$ and hence $||Q-f||_{[a,b]}<\varepsilon/2$. We will use the symbol $\Subset$ to mean compactly contained. Let $V_1\Subset V$ be a sufficiently small, $\mathbb{R}$-symmetric domain containing $f(I)$ so that there is a component of $Q^{-1}(V_1)$ (which we denote by $U_1$) satisfying $U_1\Subset U$. Let $U_2$ be a $\mathbb{R}$-symmetric, analytic domain satisfying $I\Subset U_1\Subset U_2\Subset U$. 
Recall Notation \ref{last_section_notation} and consider: 
\begin{enumerate}
\item the compact set $\overline{U_1}$, 
\item the analytic function $Q$,
\item the analytic domain $U_2$ containing $\overline{U_1}$,
\item $\textrm{min}\{\varepsilon/2,\dist(\partial V_1, \partial V) \}$,
\item $P=\{\infty\}$.
\end{enumerate}
Applying Definition \ref{g_final_defn} to (1)-(5) yields quasiregular mappings $g_n$ with $\mathbb{R}$-symmetric Beltrami coefficient, so that
\begin{equation}\label{pntoberescaled} p_n:=g_n\circ\phi_n^{-1} \end{equation}
is a real polynomial approximant of $Q$ satisfying:
 \begin{enumerate}  \item $||p_n-f||_{[a,b]}<\varepsilon$,
  \item $\textrm{CP}(p_n)\subset U_2$, 
  \item $ \textrm{CV}(p_n) \subset V$, \end{enumerate}
  for large $n$. Thus $p_n$ satisfies the conclusion of Corollary \ref{Weierstrass} for large $n$.
 \end{proof}
 

Recall the notation $\Omega(i)$ from Definition \ref{g_final_defn}, and let $\tau_i:\Omega(i)\rightarrow\mathbb{D}^*$ be the conformal mapping satisfying $\tau_i^{-1}(\infty)=P'\cap\Omega(i)$ as in Notation \ref{omega_def_setup}. The following fact justifies part of our description in the introduction of the behavior of the rational approximants off $K$.

\begin{prop}\label{largeoutsidesmalldisc} Let $1<r<R<\infty$. Then, for all sufficiently large $n$, we have
\begin{equation}\label{firstWTSoutside} r_n\circ\phi_n(z)=\tau_i(z)^m \textrm{ and } \end{equation}
\begin{equation}\label{2ndWTSoutside} |r_n(z)|>R \end{equation}
for all $z\in\tau_i^{-1}(\{z : |z|>r \})$.
\end{prop}

\begin{proof} Fix $r$ and $R$ as in the statement. From (\ref{zmoutsidenbhd}) and the functional equation (\ref{definition_of_g_notprelim}) defining $g_n$ in $\Omega(i)$, it follows that:
\begin{equation} g_n(z)=\tau_i(z)^m \textrm{ for all } z\in\tau_i^{-1}(\{z : |z|>(r+1)/2 \}) \end{equation} 
for all large $n$. Since
\begin{equation} r_n\circ\phi_n=g_n, \end{equation}
The relation (\ref{firstWTSoutside}) follows. Moreover, we have by Proposition \ref{phi_close_to_id_prop} that:
\begin{equation} \phi_n\circ\tau_i^{-1}(\{z : |z|>r \}) \subset \tau_i^{-1}(\{z : |z|>(r+1)/2 \})  \end{equation}
for all sufficiently large $n$. Since $((r+1)/2)^m>R$ for large $n$, the relation (\ref{2ndWTSoutside}) also follows.
\end{proof}

\begin{rem} If we make further assumptions on $f$ and $K$, the conclusion
\begin{equation}\nonumber \textrm{CV}(r) \subset N_\varepsilon\widehat{f(K)} \end{equation}
of Theorems  \ref{main_theorem} and \ref{rational_runge} can be improved to 
\begin{equation}\label{improvedlocation'}  \textrm{CV}(r) \subset \widehat{f(K)}, \end{equation}
which is equivalent to $\textrm{CV}(r) \subset f(K)$
if $f(K)$ is full. Indeed, if for instance the interiors of $K$, $f(K)$ are analytic domains and $f: \textrm{int}(K) \rightarrow \textrm{int}(f(K))$ is proper, then a similar strategy as in the proofs of Theorems \ref{main_theorem} and \ref{rational_runge} but replacing $\Psi$ in (\ref{g_n_disc_defn}) with a conformal map $\mathbb{D}\mapsto\textrm{int}(K)$ can be used to prove (\ref{improvedlocation'}).
\end{rem}

 \begin{rem} 
We remark that while Theorem \ref{main_theorem} strictly improves on Runge's Theorem 
on polynomial approximation, the relationship between
Theorem \ref{rational_runge} and Runge's Theorem on rational
approximation is more subtle. Both show existence of rational
approximants, and only Theorem \ref{rational_runge} describes 
the critical point structure of the approximant, however the 
poles of the approximant in Theorem \ref{rational_runge} are
specified only up to a small perturbation, whereas in Runge's
Theorem they are specified exactly. We do not know whether it 
is necessary to consider perturbations of $P'$ in Theorem 
\ref{rational_runge}, or if the improvement $r^{-1}(\infty)=P'$
is possible (a related problem appears in \cite{MR4023391},
\cite{MR4099617}, \cite{2022arXiv220202410B}, where it is
known no such improvement is possible). 
 \end{rem}



\bibliographystyle{alpha}



\end{document}